\documentclass[11pt,twoside,english]{article}
\usepackage{amsmath}
\usepackage{amsfonts}
\usepackage{mathrsfs}
\usepackage{color}
\usepackage{ amsmath, amsfonts, amssymb, amsthm, amscd}
\usepackage[T1]{fontenc}
\usepackage[english]{babel}
\usepackage[noinfoline]{imsart}

\newtheorem{theorem}{Theorem}
\newtheorem{lemma}{Lemma}
\newtheorem{proposition}{Proposition}

\newtheorem{definition}{Definition}
\newtheorem{corollary}{Corollary}
\newtheorem{remark}{Remark}
\newcommand{\be}{\begin{equation}}
\newcommand{\ee}{\end{equation}}
\newcommand{\beq}{\begin{eqnarray}}
\newcommand{\eeq}{\end{eqnarray}}
\newcommand{\ced}{\end{proof}}

\begin{document}

\begin{frontmatter}

\title{Neumann Boundary Problem for Parabolic Partial Differential Equations with Divergence Terms}
\date{}
\runtitle{}
\author{\fnms{Xue}
 \snm{YANG}\corref{}\ead[label=e1]{xyang2013@tju.edu.cn}}
%\thankstext{T1}{The work of the first author is supported by National Natural Science Foundation of China (11401427) }
\address{Tianjin University
\\\printead{e1}}
\author{\fnms{Jing}
 \snm{ZHANG}\corref{}\ead[label=e2]{zhang\_jing@fudan.edu.cn}}
%\thankstext{T2}{The work of the second author is supported by National Natural Science Foundation of China (11401108) and Shanghai Science and Technology Commission Grant (14PJ1401500).}
\address{Fudan University
\\\printead{e2}}

\runauthor{X. Yang and J. Zhang}

\begin{abstract}
We prove an existence and uniqueness result for Neumann boundary problem of a parabolic partial differential equation (PDE for short)  with a singular nonlinear divergence term which can only be understood in a weak sense. A probabilistic approach is applied by studying the backward stochastic differential equations (BSDEs for short) corresponding to the PDEs, the solution of which turns out to be a limit of a sequence of BSDEs constructed by penalization method.
\end{abstract}

\begin{keyword}
\kwd{ stochastic partial differential equations, penalization method, It\^o's formula, backward stochastic differential equations, martingale decomposition, reflecting diffusions, probabilistic representation }
\end{keyword}
\begin{keyword}[class=AMS]
\kwd[Primary ]{60H15; 35R60; 31B150}
\end{keyword}

\end{frontmatter}

\section{Introduction}

 We consider the following partial differential equation
\begin{equation}\label{PDE}
\left\{ \begin{split}
&\partial_t u(t,x)+ \frac{1}{2}\Delta u(t,x)+ \langle b, \nabla u\rangle-divg(t,x,u,\nabla u)+f(t,x,u,\nabla u)=0, (t,x)\in [0,T]\times D\\
&u(T,x)=\Phi(x), \ x\in D,\\
&\frac{\partial u}{\partial\vec{n}}(t,x)-2\langle g(t,x,u,\nabla u), \vec{n}\rangle+h(t,x,u)=0, \ (t,x)\in \ [0,T]\times \partial D,
\end{split}
\right.
\end{equation}
where $D$ is a smooth bounded domain in $\mathbb{R}^N$ endowed with the inner product $\langle,\rangle$.  $\vec n$ is the unit inward normal vector field of $D$ on the boundary $\partial D$. $f$, $g$ and $h$ are nonlinear measurable functions. $b$ is a Lipschitz continuous  $\mathbb{R}^N$-valued function.

\vspace{2mm}
This article is devoted to solving the nonlinear PDE with Neumann boundary condition by studying the BSDE corresponding to the PDE, for which the underlying process is a reflecting diffusion in domain $D$. A singular term $''div g''$ involved in the equation will be understood as a distribution, and a classic weak solution is considered in this paper.

\vspace{2mm}
The theory of nonlinear BSDEs was firstly introduced by Pardoux and Peng (\cite{PP 90}) who gave a probabilistic formula, known as generalized Feymann-Kac formula, for solving nonlinear PDEs (\cite{P 96}). Subsequently, BSDEs as useful tools in solving nonlinear problems were further studied by Pardox and Peng (\cite{PP 92},\cite{PP}), El Karoui (\cite{Karoui}) et al.. Elliptic PDEs defined on a domain with Dirichlet and Neumann boundary conditions were studied by Darling, Pardoux (\cite{DP97}) and Hu (\cite{H93}) respectively. In \cite{H93}, the boundary condition was homogeneous, and the nonlinear case was studied by Pardoux and Zhang in \cite{PZ} in which a new class of BSDEs involving an integral with respect to a continuous increasing process was studied. Pardoux and Zhang's work is one of the motivations of our present paper. We also want to mention the work of Boufoussi and Casteren (\cite{BC04}). In \cite{BC04}, they provided an approximation result of the solution of semilinear PDEs with nonlinear Neumann boundary conditions via BSDEs, and the convergence happened in S-topology (\cite{JAK97}). But both of the two works (\cite{PZ}, \cite{BC04}) focused on the viscosity solutions of the corresponding PDEs while we are interested in obtaining the weak solutions for the PDEs.

\vspace{2mm}
Not only to the viscosity solutions, BSDEs were applied to the weak solutions of PDEs, under additional regularity assumptions, by Barles and Lesigne (\cite{BL97}), Lejay (\cite{L02a}, \cite{L04}), Stoica(\cite{S}), Rozkosz(\cite{Roz03}) et al.. The notion of weak
solutions provides a natural framework for BSDEs, and the Sobolev space in which  weak solutions live or converge can be treated as a Dirichlet space, so that the decomposition and stochastic calculus can be used in the framework of Dirichlet forms (\cite{FOT}).

\vspace{2mm}
In this article, we deal with the reflecting diffusion in domain $D$ as underlying process, which can be approximated by a sequence of penalized diffusions (\cite{LS}). According to this penalization method, we construct a sequence of penalized PDEs which are not restricted by any boundary conditions but still involve the divergence terms.  Thanks to \cite{DS04}, the existence of weak solutions for these PDEs has been proved, but it is not easy to obtain the convergence of this sequence of solutions in the Sobolev space by analytic method.  According to this observation, the BSDEs involving forward-backward martingale integration (\cite{S}) connecting to the penalized PDEs are considered.   This approximation result of Neumann boundary problem with probabilistic approach is also a contribution of this article.

\vspace{2mm}
 Dealing with this singular term is a difficult point in our study, which is actually substituted by a function in Dirichlet space in our paper, so that the Fukushima decomposition can be applied.  This transformation supplies an equivalent PDE without the divergence term so that the penalization method we mentioned before can be applied.   The convergence of BSDEs connecting to be penalized PDEs  gives us a candidate solution for the PDE with Neumann boundary conditions. By the theory of Dirichlet form, we find that the candidate is a mild solution, and prove that this mild solution is also a weak solution.

\vspace{2mm}
In this paper, the Neumann boundary problem with nonlinear coefficients is proved by two steps. We firstly solve the linear PDE by penalization method. Based on this linear result, the nonlinear case is solved by Picard iteration. We use both analytic and probabilistic methods independently to calculate this approximation. 

\vspace{2mm}
The paper is organized as follows. In Section 2, we recall the decomposition of the reflecting diffusions, the penalization approximation, and some estimate results. Section 3 gives the probabilistic interpretation of the divergence term when the underlying process is a reflecting diffusion. Section 4 is devoted to studying the BSDEs containing the integration w.r.t. local time and forward-backward martingale integration, which are associated with the PDEs with Neumann boundary conditions. In Section 5, we prove the sequence of BSDEs associated with penalized PDEs is convergent and solve the linear PDE. Nonlinear Neumann problem is finally solved in Section 6.

\section{Preliminaries}

\subsection{Notations}

The domain $D\subset \mathbb{R}^N$ is bounded with smooth boundary and we assume there is a smooth function $\psi$ such that
$$
D=\{x\in \mathbb{R}^N| \psi(x)>0\}\quad\mbox{and}\quad \partial D=\{x\in \mathbb{R}^N|\psi(x)=0\}.
$$
On $\partial D$, $\vec{n}:=\nabla \psi$ coincides with the unit vector pointing inward the interior of $D$.
Set function $d(x):=d(x,\bar{D})^2$ in a neighborhood of $\bar{D}$, then $d(x)=0$ if $x\in\bar{D}$ and $d(x)>0$ otherwise. The penalization term $\vec\delta(x):=\nabla d(x)$ satisfies $\langle \nabla \psi (x), \vec\delta(x)\rangle \leq 0$, for all $x\in \mathbb{R}^N$.  Let $dx$ denote the $N-$dimensional Lebesgue measure on $\mathbb{R}^N$ and $d\sigma(x)$ the $(N-1)$-dimensional Lebesgue measure on $\partial D$.
\vspace{2mm}

$L^2(D)$ is the space of square integrable functions on $D$ with the inner product and norm as follows

$$(f,g):=\int_D f(x)g(x)dx\quad\mbox{and}\quad \|f\|^2:=(f,f).$$

For two vector valued functions $q=(q_1,\cdots,q_n)$ and $p=(p_1,\cdots,p_n)$, where $q_i, p_i\in L^2(D), i=1,...,n$,
we also use the notation $(p,q):=\int_D \sum_{i=1}^{n}q_i(x)p_i(x)dx$ for simplicity.

Let $(\mathcal{F},\mathcal{E})$ be the Dirichlet form on $L^2(D)$ associated with the operator $L_0=\frac{1}{2}\Delta$ with null Neuman boundary condition defined as
$$
\mathcal{E}(u,u)=\frac{1}{2}(\nabla u, \nabla u  ),\quad  u\in \mathcal{F},
$$
where $\mathcal{F}$ is the closure of $C^\infty(D)$ under the norm $\|\cdot\|^2_{\mathcal{E}_1}:=\mathcal{E}(\cdot,\cdot)+(\cdot,\cdot)$.  Then $\mathcal{F}$ is a Hilbert space with the norm $\|\cdot\|_{\mathcal{E}_1}$. It is well known that $\mathcal{F}=H^1(D)$ is the first order Sobolev space.

Let $L^2(\partial D)$ be the space of square integral functions on $\partial D$ with respect to Lebesgue measure $d\sigma(x)$. We denote the trace operator $Tr:H^1(D)\rightarrow L^2(\partial D)$ with the norm $\|Tr\|$.

\vspace{2mm}
Suppose the measurable functions
$$f: \mathbb{R}^{+}\times\mathbb{R}^N\times \mathbb{R}\times \mathbb{R}^N\rightarrow \mathbb{R}\quad\mbox{and}\quad h: \mathbb{R}^{+}\times\mathbb{R}^N\times \mathbb{R}\rightarrow \mathbb{R}$$
satisfy the following conditions: there exist positive constants $\alpha,\beta, K, C$, for any $x,x'\in D$, $y,y'\in \mathbb{R}$, $z,z'\in \mathbb{R}^N$,
\vspace{3mm}

\textbf{(H1)} 
$$(y-y')(f(t,x,y,z)-f(t,x,y',z))\leq \alpha|y-y'|^2,$$
$$(y-y')(h(t,x,y)-h(t,x,y')) \leq -\beta|y-y'|^2.$$\\
\textbf{(H2)} 
$|f(t,x,y,z)|\leq K(1+|y|+|z|)  \quad\mbox{and}\quad |h(t,x,y)|\leq  K.$\\
\textbf{(H3)} $y\rightarrow (f(t,x,y,z), h(t,x,y))$ is continuous for all $(t,x,z)$ $ a.e.$.\\
\textbf{(H4)} $|f(t,x, y,z)-{f}(t,x',y',z')|\leq C|x-x'|+\alpha(|y-y'|+|z-z'|).$\\
\textbf{(H5)} $|h(t,x, y)-{h}(t,x',y')|\leq C|x-x'|+\beta|y-y'|.$ 
%&|g_i(t,x, y,z)-g_i(t,x,y',z')|\leq \gamma (|y-y'|+|z-z'|), \ \mbox{for}\ i=1,\cdots,N,

%with the contraction property $\sqrt{2}\gamma+\beta\|Tr\|^2<\frac{1}{2}$.
\vspace{2mm}
Suppose the measurable vector valued function
$$
g=(g_1,\cdots,g_N): \mathbb{R}^{+}\times D \times \mathbb{R}\times \mathbb{R}^N\rightarrow \mathbb{R}^N,
$$
satisfies the Lipschitz condition: there exists a positive constant $\gamma$, for any $y,y'\in \mathbb{R}$, $z,z'\in \mathbb{R}^N$,

\textbf{(H6)} $|g_i(t,x, y,z)-g_i(t,x,y',z')|\leq \gamma (|y-y'|+|z-z'|), \ \mbox{for}\ i=1,\cdots,N.$
%The coefficient {\red $g\in C^1_b([0,T])\otimes C(\bar{D})$} which means that the term $div g=\sum_{i=1}^N \partial_i g_i$ can only be understood in weak sense. Furthermore, we extend $g$ to the whole space $\mathbb{R}^N$ by letting it equal to 0 on $\bar{D}^c$.

\vspace{3mm}
We also assume the following integrability conditions hold 
\begin{equation}\label{integrability condition}
\int_{0}^{T}\int_{D}|f(t,x,0,0)|^2+\sum_{i=1}^N |g_i(t,x,0,0)|^2 dxdt+\int_0^T\int_{\partial D}|h(t,x,0)|^2 d\sigma(x)dt< +\infty.
\end{equation}

When variables $(x,y,z)$ do need to be specified,  we use $g_t, f_t, h_t$ to denote the coefficients  sometimes for simplicity in the following discussion. 

\begin{definition}\label{defweaksol}
A function $u\in L^2([0,T]; H^{1}(D))$ is said to be a weak solution of PDE \eqref{PDE}
%\begin{equation}
%\label{PDE}
%\left\{ \begin{split}
%&\partial_t u+ \frac{1}{2}\Delta u-div(g(t,\cdot,u,\nabla u))+f(t,x,u,\nabla u)=0, \quad\mbox{on}\ [0,T]\times D\\
%&u(T,x)=\Phi(x), \\
%&\frac{\partial u}{\partial\vec{n}}(t,x)-{\red 2\langle g, \vec{n}\rangle}+h(t,x,u)=0, \quad\mbox{on}\ [0,T]\times \partial D,
%\end{split}
%\right.
%\end{equation}
if  for any test function $\phi\in \mathcal{C}^\infty(\mathbb{R}^+)\otimes \mathcal{C}^\infty(D)$,
\begin{equation}\label{weaksol}
\begin{split}
(u_T,\phi_T)-(u_0,\phi_0)&-\int_0^T(u_t,\partial_t\phi_t)dt=\int_{0}^{T}\mathcal{E}(u_t,\phi_t)dt-\int_{0}^{T}\int_D\langle b,\nabla u_t\rangle(x)\phi_t(x)dxdt\\
& \ \ \ \ \ \ \ \ \ \ \ \ \ -\int_{0}^{T} (g_t,\nabla\phi_t)dt-\int_{0}^{T}(f_t,\phi_t)dt+\int_{0}^{T}\int_{\partial D} h_t\phi_td\sigma dt.
\end{split}
\end{equation}
\end{definition}

The following analytic result will be used in the later discussion (see Chapter 8 in \cite{GT}).

\begin{proposition}
\label{regularity G}
For $g \in L^q(\mathcal{O})$, where $\mathcal{O}\subset \mathbb{R}^N$ is bounded and $q>N$,  there exists a unique weak solution $G\in H^1_0(\mathcal{O})$ for the following equation 
\begin{eqnarray*}
\Delta G-G=divg.
\end{eqnarray*}
Furthermore, $G$ is uniformly bounded, i.e. $\sup\limits_{\mathcal{O}}|G|\leq C\|g\|_{L^q}$, where $C=C(N,q,|\mathcal{O}|)$.\\
If we suppose $g\in L^\infty(\mathcal{O})$ and $\mathcal{O}$ is a $C^{1,1}-$domain, then $G\in C^{1,1}(\bar{\mathcal{O}})$, i.e. there exists a constant $C>0$, for any $x,x'\in \bar{\mathcal{O}}$, $|G(x)-G(x')|\leq C |x-x'|$.
\end{proposition}

\begin{remark}
\label{regularity of G}
Actually, given $g\in L^\infty (D)$,  we can find a bounded domain $\mathcal{O}$ with smooth boundary, such that $D\subset\subset \mathcal{O}$ and extend $g$ on $\mathcal{O}$ such that $g\in L^\infty(\mathcal{O})$. Therefore, there exists a H\"{o}lder continuous function $G\in H^1_0(\mathcal{O})$, for any test function  $\phi\in C_0^\infty(\mathcal{O})$,
\begin{eqnarray*}
\int_\mathcal{O} \langle g,\nabla\phi\rangle(x)dx=\int_\mathcal{O} \langle \nabla G,\nabla\phi\rangle(x)+ G(x)\phi(x)dx.
\end{eqnarray*}
By the uniqueness of Reisz representation theorem, we find $G\in H^1(D)$ restricted on $D$  such that $div (\nabla G)-G=div(g)$ in weak sense and $G\in C^{1,1}(D)$.
\end{remark}

\subsection{Approximation of a reflected diffusion process}
$\{B_t, t\in [0,T]\}$ is a $N-$dimensional Brownian motion on a probability space $(\Omega, \mathcal{F}, P)$. For $t\in[0,T]$, $\mathcal{F}_t$ is the $\sigma-$field 
$\sigma(B_s, s\leq t)$ augmented with the $P-$null sets of $\mathcal{F}$.

Let $b:\mathbb{R}^N\rightarrow \mathbb{R}^N$  be uniformly bounded and satisfy the Lipschitz condition, i.e. there exists a constant $C_0>0$, such that $\forall x,x'\in \mathbb{R}^N$,
$$|b(x)-b(x')|\leq C_0 |x-x'|.$$

For $n\in\mathbb{N}^*$, the diffusion process $\{X^n_t, t\in[0,T]\}$ taking values in $\mathbb{R}^N$ satisfies the following equation
\begin{equation}\label{appro.diffusion}
\left\{\begin{split}
&dX^n_t=dB_t+b(X^n_t)dt+(-n\vec\delta(X^n_t))dt,\\
&X^n_0=x\in \bar{D}.
\end{split}\right.
\end{equation}

It is well known that (see \cite{LS}), when $n$ tends to $+\infty$, $\{X^n_t, t\in[0,T]\}$ converges to the reflected diffusion $\{X_t, t\in[0,T]\}$ with the local time $\{L_t,t\in[0,T]\}$, i.e.
\begin{equation}
\label{reflecting diffusion}
\left\{\begin{split}
&dX_t=dB_t+b(X_t)dt+\vec n(X_s)dL_s,\\
&X_0=x\in\bar{D}, \quad L_t=\int_{0}^t I_{\{X_s\in \partial D\}}ds.
\end{split}\right.
\end{equation}
The following propositions will be used later. One can refer to Proposition 3.1 and 3.2 in \cite{PZ}.
\begin{proposition}
\label{appro.diffusion estimate}
(1) For every $p\in[1,\infty)$, $\lim\limits_{n\rightarrow \infty}\sup\limits_{x\in \bar{D}}E^x[\sup\limits_{s\in[0,T]}|X^n_s-X_s|^p]=0$.

(2) Set $K^n_t=-\int_0^t n\vec\delta(X^n_s)ds $, $K_t=\int_0^t\vec n(X_s)dL_s$, then
\begin{eqnarray*}
\lim\limits_{n\rightarrow\infty}E\left[\sup\limits_{t\in[0,T]}|K^n_t-K_t|^p\right]=0,\quad \forall p\in[1,\infty)
\end{eqnarray*}
and $\forall\phi\in \mathcal{C}^1_b(\mathbb{R}^N)$,
\begin{eqnarray*}
\lim\limits_{n\rightarrow\infty}E\left[\sup\limits_{t\in[0,T]}\left|\int_0^t \phi(X^n_s)dK^n_s-\int_0^t \phi(X_s)dK_s\right|\right]=0.
\end{eqnarray*}
\end{proposition}
\begin{proposition}\label{boundLt}
For all $p\geq 1$, there exists a constant $C_p$ such that $\forall(t,x)\in [0,\infty)\times \bar{D}$,
$$
E^x[|L_t|^p]\leq C_p(1+t^p)
$$
and for each $\mu,t>0$, there exists $C(\mu,t)$ such that $\forall x\in \bar{D}$,
$$
E^x(e^{\mu L_t})\leq C(\mu, t).
$$
\end{proposition}
\begin{corollary}
\label{bound.local time}
For any $x\in\bar{D}$,
$$
E^x\left[\int_{0}^{T}e^{\mu L_t}dt+\int_0^T e^{\mu L_t}dL_t\right]<+\infty.
$$
\end{corollary}
\begin{proof}
Since the local time $\{L_t\}_{t\in[0,T]}$ is increasing, it follows that
\begin{equation*}
\begin{split}
E^x\left[\int_{0}^{T}e^{\mu L_t}dt+\int_0^T e^{\mu L_t}dL_t\right]
&\leq T E^x[e^{\mu L_T}]+E^x[e^{2\mu L_T}]^{\frac{1}{2}}\cdot E^x[|L_T|^2]^{\frac{1}{2}}\\
&\leq TC(\mu, T)+C(2\mu, T)^{\frac{1}{2}}(C_2(1+T^2))^{\frac{1}{2}}<+\infty, 
\end{split}
\end{equation*}
where the second inequality comes from Proposition \ref{boundLt}. 
\end{proof}

\section{Interpretation of the Divergence Term}
In this section, we will give a stochastic representation for the divergence term in $(\ref{PDE})$ expressed as a measurable field. The second order operator in  $(\ref{PDE})$ is nonsymmetric with Neumann boundary condition, then it is associated with a reflecting diffusion.

The bilinear form
 $$\mathcal{E}(u,v)=\frac{1}{2}\int_D \sum_{i=1}\frac{\partial u}{\partial x_i}\frac{\partial v}{\partial x_i}dx,\quad \forall u,v\in H^1(D)$$
is associated with the generator $L_0=\frac{1}{2}\Delta$ satisfying the Neumann boundary condition $\frac{\partial u}{\partial\vec{n}}=0$ on $\partial D$. Set the operator $Lu:=L_0 u+\langle b,\nabla u\rangle$. Then $L$ generates a semigroup $(P_t)_{t\geq 0}$ which possesses continuous densities $\{p(t,x,y),t\geq 0, x,y\in \bar{D}\}$.  It is well known that the reflecting diffusion ($\ref{reflecting diffusion}$) is associated with operator $L$,
and for any $u\in H^1(D)$, the Fukushima decomposition(\cite{FOT}) is as follows
$$
u(X_t)-u(X_s)=Mu|^t_s+Nu|^t_s\,,
$$
where $Mu|^t_s:=\int_s^t \langle \nabla u(X_r),dB_r\rangle$ is the martingale additive functional and $Nu|^t_s$ is the zero-energy additive functional. 
For $u\in C^2(\bar D)$,
$$
Nu|^t_s:=\int_s^t Lu(X_r)dr+\int_s^t \frac{\partial u}{\partial\vec{n}}(X_r)dL_r\,,
$$
where $L_t$ is the additive functional corresponding to the Lebesgue measure $\sigma(x)$ on $\partial D$. It follows that
$$
E^x[\int_{0}^{t}f(X_r)dL_r]=\int_{0}^{t}\int_{\partial{D}} p(r,x,y)f(y)\sigma(dy)dr.
$$

Consider the reverse process $(X_{T-t})_{t\in[0,T]}$ under the probability $P^{o}$, for $o\in\bar D$, with the non-homogenous transition function
\begin{eqnarray*}
Q_{0,t}u(x)=\frac{\int_D p(T-t,o,y)u(y)p(t,y,x)dy}{p(T,o,x)}.
\end{eqnarray*}
We denote the density of $Q_{0,t}$ by $p_Q(t,x,y)=\frac{p(T-t,o,y)p(t,y,x)}{p(T,o,x)}$.
%Though
%\begin{eqnarray*}
%u(X_{T-t})-u(X_T)=\int_{s}^t \langle \nabla u(X_r),dB_r\circ r_T+\int_0^{t}Lu(X_{T-r})dr+\frac{1}{2}\int_{T}^{T-t}\frac{\partial u}{\partial n}(X_r)dL_r
%\end{eqnarray*}

By the methods in Propostion 3.1 of \cite{LS}, we obtained the following results associated with reflecting diffusions.
\begin{lemma}
Fix $o\in \bar{D}$ and set $p_t(x)=p(t,o,x)$, then
\begin{equation*}\begin{split}
Q_{0,t}u-u=\int_{0}^{t}Q_{0,r}(\frac{1}{2}\Delta u-\langle b,\nabla u\rangle
&+\frac{\langle \nabla p_{T-r},\nabla u\rangle}{p_{T-r}})dr\\&+\frac{1}{2}\int_{0}^{t}\int_{\partial{D}}p_Q(r,x,y)\frac{\partial u}{\partial \vec n}(y)\sigma(dy)dr.
\end{split}\end{equation*}
\end{lemma}

\begin{proof}
%\begin{eqnarray}
%&&p(T,o,x)[\int_{0}^{t}Q_{0,r}(\frac{1}{2}\Delta u-\langle b,\nabla u\rangle)dr
%+\int_{0}^{t}\int_{\partial}p_Q(t,x,y)\frac{\partial u}{\partial n}(y)\sigma(dy)dr\nonumber\\
%&=&\int_{0}^{t}p(T-r,o,y)L^{*}u(y)p(r,y,x)dydr\nonumber\\
%&=&\int_{0}^{t}\int_D L(p(T-r,o,y)p(r,y,x))u(y)dydr\nonumber\\
%&=&\int_{0}^{t}\int_D p(T-r,o,y)[Lp(r,y,x)]u(y)dydr-\int_{0}^{t}\int_D [L^{*}p(T-r,o,y)]p(r,y,x)u(y)dydr\nonumber\\
%&&\quad -\int_{0}^{t}\int_D \langle \nabla p(T-r,o,y),\nabla u(y)\rangle p(r,y,x)dydr\nonumber\\
%&=&2[\int_D p(T-t,0,y)u(y)p(t,y,x)dx-p(T,o,x)u(x)] -\int_{0}^{t}\int_D \langle \nabla p(T-r,o,y),\nabla u(y)\rangle p(r,y,x)dydr
%\end{eqnarray}

\begin{equation*}\begin{split}
&p_T(x)\int_{0}^{t}Q_{0,r}(\frac{1}{2}\Delta u-\langle b,\nabla u\rangle)dr\\
=&\int_{0}^{t}\int_D p(T-r,o,y)(\frac{1}{2}\Delta u-\langle b,\nabla u\rangle)(y)p(r,y,x)dy\\
=&-\int_{0}^{t}\int_D L^*p_{T-r}(y)u(y)p(r,y,x)dydr+\int_{0}^{t}\int_D L^*p(r,y,x)p(T-r,o,y)u(y)dydr\\
&-\int_{0}^{t}\int_D \langle \nabla p_{T-r},\nabla u\rangle p(r,y,x)dydr
-\frac{1}{2}\int_0^t\int_{\partial D}p(T,o,y)\frac{\partial u}{\partial \vec{n}}(y)p(r,y,x)\sigma(dy)dr\\
=& \int_D p(T-t,o,y)u(y)p(t,y,x)dy-p_T(x)u(x)-\int_{0}^{t}\int_D \langle \nabla p_{T-r},\nabla u\rangle p(r,y,x)dydr\\
&-\frac{1}{2}\int_0^t\int_{\partial D}p(T,o,y)\frac{\partial u}{\partial \vec{n}}(y)p(r,y,x)\sigma(dy)dr,
\end{split}\end{equation*}
where the second equality is derived by integration by parts,  $L^*$ is the dual operator of $L$ on $L^2(D)$,  and the last equality is obtained by $L^* p(t,o,y)=\partial_t p(t,o,y)$.
\end{proof}

\begin{proposition}
Fix $o\in \bar{D}$ and set the following process
\begin{equation}\begin{split}
\bar{M}u|^T_{T-t}:=u(X_{T-t})&-u(X_T)-\int_{0}^{t}(\frac{1}{2}\Delta u-\langle b,\nabla u\rangle)(X_{T-r})dr
\\&-\int_{0}^{t}\frac{\langle \nabla p_{T-r},\nabla u\rangle(X_{T-r})}{p_{T-r}(X_{T-r})}dr-\int_{T-t}^{T}\frac{\partial u}{\partial \vec n}(X_r)dL_r.
\end{split}\end{equation}
%where $\tilde{L}_t=\int_{0}^t I_{\partial D}(X_{T-r})d\tilde{L}_r$, and
%$$\int_{0}^{t}\frac{\partial u}{\partial n}(X_{T-r})d\tilde{L}_r=\int_{T-t}^{T}\frac{\partial u}{\partial n}(X_r)dL_r.$$
\vspace{2mm}

(1). $\{\bar{M}u|^T_{T-t}\}_{t\in[0,T]}$ is a martingale with respect of the filtration $\mathcal{F}^{'}_t=\sigma\{X_{T-s},s\in[0,t]\}$
and $$\bar{M}u|^T_t-\bar{M}u|^T_s=\bar{M}u|^s_{t}.$$
(2). The following relation holds:
$$u(X_t)-u(X_0)=\frac{1}{2}Mu|^t_0-\frac{1}{2}(\bar{M}u|^T_0-\bar{M}u|^T_t)+\int_{0}^{t}\langle b,\nabla u\rangle (X_r)dr
-\frac{1}{2}\int_{0}^{t}\frac{\langle \nabla p_{r},\nabla u\rangle(X_{r})}{p_{r}(X_{r})}dr.$$
\end{proposition}

\begin{proof}
Since
\begin{equation*}\begin{split}
\bar{M}u|^T_{t}:=u(X_{t})-u(X_T)&-\int_{0}^{T-t}(\frac{1}{2}\Delta u-\langle b,\nabla u\rangle)(X_{T-r})dr
\\&-\int_{0}^{T-t}\frac{\langle \nabla p_{T-r},\nabla u\rangle(X_{T-r})}{p_{T-r}(X_{T-r})}dr-\int_{t}^{T}\frac{\partial u}{\partial \vec n}(X_r)dL_r, 
\end{split}
\end{equation*}
%and
%\begin{eqnarray}
%\bar{M}f|^T_{0}:&=&u(X_{0})-u(X_T)-\int_{0}^{T}(\frac{1}{2}\Delta u-\langle b,\nabla u\rangle)(X_{T-r})dr
%-\int_{0}^{T}\frac{\langle \nabla p_{T-r},\nabla u\rangle(X_{T-r})}{p_{T-r}(X_{T-r})}dr\nonumber\\
%&&-\int_{0}^{T}\frac{\partial u}{\partial n}(X_{T-r})d\tilde{L}_r£¬
%\end{eqnarray}
it follows that 
\begin{equation*}\begin{split}
\bar{M}u|^T_{t}-\bar{M}u|^T_{s}&=u(X_{t})-u(X_{s})-\int_0^{s-t} (\frac{1}{2}\Delta u-\langle b,\nabla u\rangle)(X_{s-r})dr\\
&-\int_{0}^{s-t}\frac{\langle \nabla p_{s-r},\nabla u\rangle(X_{s-r})}{p_{s-r}(X_{s-r})}dr
-\int_{t}^{s}\frac{\partial u}{\partial \vec n}(X_{r})dL_r=\bar{M}u|^s_t
\end{split}
\end{equation*}
and
\begin{equation*}\begin{split}
u(X_t)-u(X_0)=&\,\bar{M}u|^T_t-\bar{M}u|^T_0-\int_{T-t}^{T}(\frac{1}{2}\Delta u-\langle b,\nabla u\rangle)(X_{T-r})dr
\\&-\int_{0}^{t}\frac{\partial u}{\partial \vec n}(X_{r})d{L}_r
-\int_{T-t}^{T}\frac{\langle \nabla p_{T-r},\nabla u\rangle(X_{T-r})}{p_{T-r}(X_{T-r})}dr\\
=&-\bar{M}u|^t_0-\int_{0}^{t}(\frac{1}{2}\Delta u-\langle b,\nabla u\rangle)(X_{r})dr
\\&-\int_{0}^{t}\frac{\partial u}{\partial \vec n}(X_{r})d{L}_r-\int_{0}^{t}\frac{\langle \nabla p_{r},\nabla u\rangle(X_{r})}{p_{r}(X_{r})}dr.
\end{split}
\end{equation*}
Then
\begin{equation*}
2(u(X_t)-u(X_0))=Mu|^t_0-\bar{M}u|^t_0+2\int_{0}^{t}\langle b,\nabla u\rangle (X_r)dr
-\int_{0}^{t}\frac{\langle \nabla p_{r},\nabla u\rangle(X_{r})}{p_{r}(X_{r})}dr.
\end{equation*}

Therefore, we get the forward-backward martingale decomposition
$$u(X_t)-u(X_0)=\frac{1}{2}Mu|^t_0-\frac{1}{2}\bar{M}u|^t_0+\int_{0}^{t}\langle b,\nabla u\rangle (X_r)dr
-\frac{1}{2}\int_{0}^{t}\frac{\langle \nabla p_{r},\nabla u\rangle(X_{r})}{p_{r}(X_{r})}dr.$$
\end{proof}

\begin{corollary}
(1). For $u,v\in H^1(D)$,
$$\langle Mu, Mv\rangle_t=\int_{0}^{t}\langle \nabla u, \nabla v\rangle(X_r)dr$$ and 
$$\langle \bar{M}u|^T_\cdot, \bar{M}v|^T_\cdot \rangle_t=\int_{t}^{T}\langle \nabla u, \nabla v\rangle(X_r)dr.$$
(2). For $x=(x_1,\cdots,x_N)\in D$, set $u_i (x)=x_i$, $M^i(t)=Mu_i|^t_0$ and $\bar{M}^i(t,T)=\bar{M}u_i|^T_t$, then
$$X^i_t-X^i_0=\frac{1}{2}M^i(t)-\frac{1}{2}\bar{M}^i(0,t)+\int_{0}^{t} b_i(X_r)dr
-\frac{1}{2}\int_{0}^{t}\frac{\partial_i p(X_{r})}{p_{r}(X_{r})}dr.$$
\end{corollary}

\vspace{4mm}
For $g=(g_1,\cdots,g_N): \mathbb{R}^N\rightarrow \mathbb{R}^N$, we define the backward stochastic integral 
\begin{eqnarray*}
\int_s^t g_i(X_{r})d\bar{M}^i_t:=(L^2-)\lim_{\delta\rightarrow 0}\sum_{j=0}^{n-1} g(X_{t_{j+1}})\bar{M}^i(t_j,t_{j+1}),
\end{eqnarray*}
where the limit is over the partition $s=t_0<t_1<\cdots<t_n=t$ and $\delta=\max_j (t_{j+1}-t_j)$.

Define
\begin{equation*}
\int_s^t g\ast dX_r=\int_s^t g(X_r) dM_r+\int_s^t g(X_r)d\bar{M}_r
+\int_s^t \frac{\langle g, \nabla p_r\rangle}{p_r}(X_r)dr+2\int_s^t \langle g,\vec n\rangle(X_r)dL_r\,.
\end{equation*}

\begin{proposition}\label{decompositionofG}
For $G\in H^1(D)$, then we have the decomposition
\begin{equation*}
G(X_t)-G(X_s)=\int_s^t \langle \nabla G(X_r),dM_r\rangle+\int_s^t \langle b,\nabla G\rangle(X_r)dr+\int_s^t \frac{\partial G}{\partial \vec n}(X_r)dL_r-\frac{1}{2}\int_s^t \nabla G\,\ast\,dX_r.
\end{equation*}
\end{proposition}

The following lemma, which can be proved similarly as Lemma 3.1 of \cite{S},  is very important in  interpretation of the divergence term $div g$ in PDE $(\ref{PDE})$.
\begin{lemma}
\label{div g}
For $g\in L^2(\mathbb{R}^N;\mathbb{R}^N)$, if there is a function $G\in L^2(\mathbb{R}^N)$, such that $div g=G$ in weak sense, then
\begin{eqnarray*}
\int_s^t G(X_r)dr=-\int_s^t g\ast dX_r.
\end{eqnarray*}
\end{lemma}

%By Lemma 3.5 in $\cite{S}$, there exists a unique function $G\in L^2([0,T]\times D)$ such that $G(t,\cdot)\in H^1(D) $ for almost all $t$,  $\int_{0}^{T} \|\nabla G(t,\cdot)\|^2 dt<\infty$, and
%$$
%div(g)=div(\nabla G)-G
%$$
%in weak sense.

\section{Backward stochastic differential equations with $\ast-$ integral}

In this section, we suppose the divergence term $g$ only depends on $(t,x)$. We will prove that under certain conditions, the following BSDE admits a unique solution $(Y,Z)$,
\begin{equation}\label{BSDE Y}
dY_t=-f(t, X_t,Y_t,Z_t)dt-(h(t, X_t,Y_t)+2\langle g(t,X_t),\vec n\rangle )dL_t+ g(t, X_t)*dX_t+Z_tdB_t.
\end{equation}

In the following discussion,  we simply assume $b=0$ in PDE \eqref{PDE},  and consider the symmetric reflecting diffusions correspondingly. Actually, we can combined the drift term $\langle b,\nabla u\rangle$  and nonlinear term $f(t,x,u,\nabla u)$  into a new nonlinear term $F(t,x,u,\nabla u):=\langle b,\nabla u\rangle+f(t,x,u,\nabla u)$, so that this assumption is realized,  without weakening our result.  

\vspace{3mm}

\vspace{3mm}
The following lemma is obtained by Reisz representation theorem and Proposition \ref{regularity G}.
\begin{lemma}
	\label{lamma. regularity of G}
	Assume $g\in L^2([0,T]\times D;\mathbb{R}^N)$, then there exists a unique function $G\in L^2([0,T]\times D; \mathbb{R}^N)$, $G(t,\cdot)\in H^1(D)$ for almost all $t\in [0,T]$ ,
	$\int_0^T \|G(t)\|^2_{H^1(D)}dt<+\infty$, and $\forall\phi\in \mathcal{C}^\infty([0,T]\times D)$,
	\begin{eqnarray*}
		\int_0^T (G(t),\phi(t))+(\nabla G(t),\nabla \phi(t))dt=\int_{0}^T (g(t),\nabla \phi(t))dt.
	\end{eqnarray*}
Furthermore,  If $g\in C^1([0,T])\otimes L^\infty(D)$,  then  $G\in C^1([0,T])\otimes H^1(D)$ and for fixed $t\in [0,T]$, $G(t,\cdot)$ is H\"{o}lder continuous in $\bar{D}$.
\end{lemma}
\begin{proof}
	We only need to prove the second part of this lemma. If $g(t,x)=g_1(t)g_2(x)$, then it is easy to know that $G(t,x)=g_1(t)G_2(x)$ where $div G_2-G_2=div g_2$ in weak sense.  Then by Proposition \ref{regularity G} and Remark \ref{regularity of G}, the lemma is proved.
	% Since for almost every $t\in[0,T]$, we have for any $\phi\in C^\infty(D) $, $(g(t),\nabla\phi)=(G(t),\phi)_{H^1(D)}$. As $t\rightarrow (g(t),\nabla\phi)$ is in $C^1([0,T])$, we have a  version  of $G$ such that $(\partial_tG(t),\phi)_{H^1(D)}$ is also bounded in $[0,T]$. By the arbitrary of $\phi$, $G(t,x)\in C^1([0,T])$ for almost every $x\in D$.
\end{proof}

%\begin{remark}
%\label{remark. weak G}
%In the rest of this paper, we assume that $g\in C^1([0,T])\otimes C(\bar{D})$, so by Lemma \ref{lamma. regularity of G} and Remark \ref{regularity of G}, we know that $G\in C^1([0,T])\otimes H^1(D)$ and for fixed $t\in [0,T]$, $G(t,\cdot)$ is H\"{o}lder continuous in $\bar{D}$.
%\end{remark}
\begin{remark}
	\label{remark. equivalent pde}
	Lemma \ref{lamma. regularity of G} means that $div g=div(\nabla G)-G$ on $D$ and $g=\nabla G$ on $\partial D$ in weak sense.  Therefore, the weak solution $u$ of PDE (\ref{PDE}) also satisfies the following equation
	\begin{equation}\label{equivalent pde}
	\left\{
	\begin{split}
	&\partial_t u+ \frac{1}{2}\Delta u-div(\nabla G(t,\cdot))+G(t,\cdot)+f(t,\cdot,u,\nabla u)=0, \quad\mbox{on} \ [0,T]\times D,\\
	&u(T,\cdot)=\Phi(\cdot),\quad\mbox{on}\ D,\\
	&\langle \nabla (u-2G), \vec{n}\rangle+h(t,\cdot,u)=0, \quad \mbox{on}\ [0,T]\times \partial D.
	\end{split}
	\right.
	\end{equation}
\end{remark}

By the same approximation  method in Theorem 3.2 in \cite{S}, the following proposition is obtained, which gives a probabilistic interpretation of the solution .
\begin{proposition}\label{decompositionofu}
If $u$ is the weak solution of Neumann boundary problem \eqref{PDE},  the process $u(t,X_t)$ satisfies the following differential equation, for $0\leq s\leq t\leq T$,
\begin{equation}\label{expression u}
\begin{split}
u(t,X_t)-u(s,X_s)&=-\int_s^tf(r,X_r,u_r(X_r),\nabla u_r(X_r))dr+\int_s^t\langle \vec{n},\nabla u_r\rangle (X_r) dL_r\\
&-\int_s^t g(r,X_r)\ast dX_r 
+\int_s^t\langle \nabla u_r(X_r), dB_r\rangle.
\end{split}
\end{equation}
\end{proposition}
\begin{proof}
%We only consider the case that $g$ is independent of $u$ and $\nabla u$. 
%In fact, in the general case,  letting $g_u(t,x)=g(t,x,u,\nabla u)$, we will be able to apply the linear result. 

Firstly,  we'll give an estimate on the weak solution $u$ of PDE $(\ref{PDE})$. With the Lipschitz and integrability conditions, we have
\begin{equation*}
\begin{split}
\|u_t\|^2+\int_t^T \|\nabla u_r\|^2dr
&=\,\|\Phi\|^2+2\int_t^T\int_D f(r,x,u_r,\nabla u_r)u_r(x)dxdr\\&+2\int_t^T\int_D g(r,x)\nabla u_r(x)dxdr+\int_t^T\int_{\partial D} h(r,x,u_r)u_r(x)d\sigma(x)dr\\
\leq&\,\|\Phi\|^2++(2\alpha+\frac{1}{\epsilon_2}+1) \int_t^T\|u_r\|^2dr+\alpha^2\epsilon_2\int_t^T \|\nabla u_r\|^2dr\\
&+\int_t^T\|f_r(0,0)\|^2dr+\epsilon_1\int_t^T \|\nabla u_r\|^2dr+\frac{1}{\epsilon_1}\int_t^T\|g_r\|^2dr\\
&+\|Tr\|^2\epsilon_3\int_t^T \|u_r\|_{H^1}^2dr
+{\beta\|Tr\|^2}\int_t^T \|u_r\|_{H^1}^2dr\\&+\frac{1}{4\epsilon_3}\int_t^T\int_{\partial D}|h_r(x,0)|^2dxdr,
\end{split}\end{equation*}
where  $\|Tr\|$ is the norm of trace operator.
 
By further calculation, we obtain
\begin{equation*}
\begin{split}
&\|u_t\|^2+(1-\epsilon_1-\alpha^2\epsilon_2-\|Tr\|^2\epsilon_3-\beta\|Tr\|^2)\int_t^T \|\nabla u_r\|^2dr\\
\leq&\,\|\Phi\|^2+(2\alpha+\frac{1}{\epsilon_2}+1+\|Tr\|^2\epsilon_3+\beta\|Tr\|^2)\int_t^T \|u_r\|^2dr\\
&+\int_t^T\|f_r(0,0)\|^2dr+\frac{1}{\epsilon_1}\int_t^T\|g_r\|^2dr+\frac{1}{4\epsilon_3}\int_t^T\int_{\partial D}|h_r(x,0)|^2d\sigma(x)dr.
\end{split}
\end{equation*}

Since $\beta<\frac{1}{\|Tr\|^2}$, we chose $\epsilon_1, \epsilon_2,\epsilon_3$ such that $\epsilon_1+\alpha^2\epsilon_2+\|Tr\|^2\epsilon_3+\beta\|Tr\|^2<1$. Then by Gronwall's inequatily, there is a constant $C>0$ depending on $T, \alpha, \beta$, such that
\begin{equation}\label{estimate on solution}
\int_0^T \|u_t\|^2_{H^1}dt\leq C\Big(\|\Phi\|^2+\int_0^T\int_{\partial D}|h_r(x,0)|^2d\sigma(x)dr+\int_0^T\int_{D}|g_r(x)|^2+|f_r(x,0,0)|^2dxdr\Big).
\end{equation} 
%Approximating of $g$ by smooth functions in $L^2([0,T];H^1(D))$, since u is also a solution for \eqref{equivalent pde}, we obtain the following representation,
%\begin{eqnarray*}
%du(t,X_t)=div(g)(X_t)dt-G(X_t)dt-f(t,X_t,u_t)dt+\langle \vec n,\nabla u_t\rangle dL_t+\langle \nabla u_t, dB_t\rangle.
%\end{eqnarray*}
%Since $\int_s^t div(\nabla G)(X_r)dr=-\int_s^t \nabla G\ast dX_r=-\int_s^t g_r\ast dX_r+\int_s^t G(X_r)dr$,
%\eqref{equivalent pde} is proved.

Secondly, we prove the representation \eqref{decompositionofu}. Let a sequence of smooth function $G^n$ approximate $G\in L^2([0,T]; H^1(D))$ obtained in Lemma \ref{lamma. regularity of G}. We denote the solution of \eqref{equivalent pde} corresponding to  $G^n$ as $u^n$ and obtain the following representation:
\begin{equation*}\begin{split}
du^n(t,X_t)=&div(\nabla G^n)(X_t)dt-G^n(t,X_t)dt-f(t,X_t,u^n_t,\nabla u^n_t)dt\\&+\langle \vec n,\nabla u^n_t\rangle(X_t) dL_t+\langle \nabla u^n_t(X_t), dB_t\rangle.
\end{split}\end{equation*}
Since $ \int_s^t div(\nabla G^n)(X_r)dX_r=-\int_s^t \nabla G^n \ast dX_r$,  $u^n$ satisfies the decomposition:
\begin{equation}\label{expression un}
\begin{split}
u^n(t,X_t)&-u^n(s,X_s)=-\int_s^tf(r,X_r,u^n(X_r),\nabla u^n(X_r))+G^n(r,X_r)dr
\\&+\int_s^t\langle \vec{n},\nabla u^n_r\rangle(X_r) dL_r-\int_s^t \nabla G^n(r,X_r)\ast dX_r +\int_s^t\langle \nabla u^n_r(X_r), dB_r\rangle.
\end{split}
\end{equation} 
By the estimate in \eqref{estimate on solution},  we know that $u^n$ approaches to $u$ in $L^2([0,T];H^1(D))$.  Passing limits  on both sides of 
\eqref{expression un},  it is easy to check that
\begin{equation*}
\begin{split}
u(t,X_t)&-u(s,X_s)=-\int_s^tf(r,X_r,u(X_r),\nabla u(X_r))+G(r,X_r)dr\\&+\int_s^t\langle \vec{n},\nabla u_r\rangle(X_r) dL_r
-\int_s^t \nabla G(r,X_r)\ast dX_r +\int_s^t\langle \nabla u_r(X_r), dB_r\rangle.
\end{split}
\end{equation*}
Therefore, the representation \eqref{decompositionofu} is prove,  since 
$\int_s^t (\nabla G-g)\ast dX_r=-\int_s^t G(r,X_r)dr$.

\end{proof}

\begin{remark}
(1) In the following discusstion, we always assume that $g(t,x)\in C^1([0,T])\otimes L^\infty(D)$.  Since $C^1([0,T])\otimes L^\infty(D)$ is dense in $L^2([0,T]\times D)$, following the same approximation method in Proposition \ref{decompositionofu},  we will get the same result in the general case.

(2) Proposition \ref{decompositionofu} holds for general $g=g(t,x,u,\nabla u)$ by setting $g_u(t,x)=g(t,x, u(t,x),\nabla u(t,x))$.
\end{remark}

\vspace{3mm}
If $u$  is the weak solution of \eqref{PDE},  then by Proposition \ref{decompositionofG} and \ref{decompositionofu}, we have the following decomposition, for $0\leq s\leq t\leq T$,
 \begin{equation*}\begin{split}
 (u-2G)(t,X_t)&-(u-2G)(s,X_s)=\int_s^t \langle \nabla(u-2G)(X_r),dB_r\rangle+\int_s^t \langle\nabla( u-2G),\vec{n}\rangle (X_r)dL_r\\
 &-\int_s^t f(r,X_r, u(X_r), \nabla u(X_r))+\partial_r G(r,X_r)+G(r,X_r)\,dr.
\end{split} \end{equation*}
 Then, this observation gives us an idea to find the solution for $(\ref{PDE})$ by solving the following BSDE:
 \begin{equation}
 \label{transformed bsde}
 \tilde{Y_t}=\Phi(X_T)-2G(T,X_T)+\int_t^T\tilde{f}(r,X_r,\tilde{Y}_r,\tilde Z_r)dr
 +\int_t^T\tilde{h}(r,X_r,\tilde{Y}_r)dL_r-\int_t^T\langle\tilde{Z}_r,dB_r\rangle,
 \end{equation}

with
$$\tilde{f}(t,X_t,y,z)=f(t,X_t,y+2G(t,X_t),z+2\nabla G(t,X_t))+2\partial_tG(t,X_t)+G(t,X_t)$$
and $$\tilde{h}(t,X_t,y)=h(t,X_t,y+2G(t,X_t)).$$

By \cite{PZ}, the following theorem is obtained.

\begin{theorem}
\label{theorem. G}
Assume that \textbf{(H1)} $\sim$ \textbf{(H4)} hold and $\Phi$ is a continuous function on $\bar{D}$.
(1)There exists a unique solution $(\tilde Y,\tilde Z)$ satisfying the following equation:
\begin{equation}
\label{BSDE-1}
\left\{\begin{split}
&d\tilde{Y_t}=-\tilde f(t,X_t,\tilde Y_t,\tilde Z_t)dt-\tilde{h}(t,X_t,\tilde{Y}_t)dL_t +\langle\tilde{Z}_t,dB_t\rangle\\
& \tilde Y_T=\Phi(X_T)-2G(T,X_T)
 \end{split}\right.
\end{equation}
and $$E\Big[\sup_{t\in[0,T]}|\tilde Y_t|^2+\int_{0}^{T}|\tilde{Z}_t|^2dt\Big]<+\infty\,.$$

%Furthermore, for any $\mu\geq 0$, there is a constant $C$ depending on $\alpha,\beta,K,T$, such that
%\begin{equation}
%\begin{split}
%&E\Big[\sup_{t\in[0,T]}e^{\mu L_t}|\tilde Y_t|^2+\int_{0}^{T}e^{\mu L_t}|\tilde{Y}_t|dL_t
%+\int_{0}^{T}e^{\mu L_t}\|\tilde{Z}_t\|^2dt\Big] \\
%&\leq\,CE\Big[1+\int_{0}^{T} e^{\mu L_t}\big[|{f}(t,X_t,0,0)|^2
%+\int_{0}^{T}e^{\mu L_t}|h(t,X_t,0)|^2dL_t\Big].
%\end{split}
%\end{equation}
(2) Set $Y_t=\tilde Y_t+2G(t,X_t)$, $Z_t=\tilde Z_t+G(t,X_t)$, then $(Y,Z)$ is the unique solution for the BSDE
\begin{equation}
\label{BSDE-2}
\left\{\begin{split}
&dY_t=-f(t, X_t,Y_t,Z_t)dt-(h(t, X_t,Y_t)+2\langle g(t,\cdot),\vec n\rangle(X_t))dL_t+ g(t, X_t)*dX_t+Z_tdB_t.\\
&Y_T=\Phi(X_t)
\end{split}\right.
\end{equation}

\end{theorem}

\begin{proof}
(1) From \textbf{(H1)}, we know
\begin{equation*}\begin{split}
&(y-y')(\tilde{f}(t,X_t,y,z)-\tilde{f}(t,X_t,y',z))\\=&(y-y')(f(t,X_t, y+2G,z+\nabla G)-f(t,X_t,y'+2G,z+\nabla G))\leq \alpha|y-y'|
\end{split}\end{equation*}
and 
$$(y-y')(\tilde{h}(t,X_t,y)-\tilde{h}(t,X_t,y'))=(y-y')(h(t,X_t, y+2G)-h(t,X_t,y'+2G))\leq -\beta |y-y'|.$$
Since
\begin{equation*}
\begin{split}
|\tilde f(t,X_t,y,z)|&\leq\, 2|\partial_t G(t,X_t)|+(2K+1)|G(t,X_t)|+2K|\nabla G(t,X_t)|+K(|y|+|z|)\\
&:=\tilde{\varphi}(t,X_t)+K
\end{split}
\end{equation*}
and
$$
|\tilde h(t,X_t,y)|\leq K:=\tilde\psi.
$$
By  Corollary \ref{bound.local time} and the boundedness of $\partial_t G$, $\nabla G$ and $G$, we have
\begin{equation*}
E\left[\int_{0}^{T}e^{\mu L_t} |\tilde{\varphi}|^2(t,X_t)dt+\int_0^T e^{\mu L_t}|\tilde{\psi}|^2(t,X_t)dL_t\right]<+\infty.
\end{equation*}
Then, with Proposition 1.1 and Theorem 1.7 in \cite{PZ}, we get the desired result.

\vspace{2mm}
(2)  BSDE \eqref{BSDE-2} is easily be obtained by adding the decomposition of $G(t,X_t)$ in Proposition \ref{decompositionofG} to BSDE \eqref{BSDE-1} and considering the relation in Lemma \ref{div g}.   

If $(U,V)$ is another solution for \eqref{BSDE-2}, then $(U_t-2G(X_t),V_t-2\nabla G(X_t))$ is the solution for \eqref{BSDE-1}. Therefore, by the uniqueness in (1), we conclude $U=Y$ and $V=Z$.  
\end{proof}

Combining Proposition \ref{decompositionofu} and Theorem \ref{theorem. G}, we obtain the following corollary.
\begin{corollary}
If $u\in L^2([0,T];H^1(D))$ is the weak solution of PDE \eqref{PDE}, then $(u(\cdot,X_\cdot),\nabla u(\cdot,X_\cdot))$ is the unique solution for BSDE \eqref{BSDE-2}.
\end{corollary}

In the discussion of following sections, we will prove the converse argument of the last corollary and build the bi-directional relationship between BSDE \eqref{BSDE-2} and PDE \eqref{PDE}.
\section{Existence and uniqueness of solution for linear Neumann boundary problem}

In this section, we consider the PDE with linear coefficients,
\begin{equation}
\label{linearPDE}
\left\{ \begin{split}
&\partial_t u(t,x)+ \frac{1}{2}\Delta u(t,x)-div(g(t,x))+f(t,x)=0,\quad  (t,x)\in[0,T]\times D,\\
&u(T,x)=\Phi(x), \quad x\in D,\\
&\frac{\partial u}{\partial\vec{n}}(t,x)-2\langle g(t,\cdot), \vec{n}\rangle(x)+h(t,x)=0, \quad (t,x)\in[0,T]\times\partial D.
\end{split}
\right.
\end{equation}
The penalization method is applied in the following discussion. We approximate the Neumann boundary problem by a sequence of PDEs without any boundary conditions, which is constructed by the classic penalization sequence of the reflecting diffusions.  

\subsection{The penalization method and approximation result}

In this section, we will construct a sequence of $(Y^n, Z^n)$, which corresponds to the weak solution of penalized PDE (\ref{appro.pde}), turning out to converge to the pair of solution  $(Y,Z)$ corresponding to the solution for \eqref{linearPDE}.

Let $H:=L^2(\mathbb{R}^N)$ be the space of square integrable functions on  $\mathbb{R}^N$ endowed with the norm $\|u\|^2:=\int_{\mathbb{R}^N} |u(x)|^2 dx$. $\frac{1}{2}\Delta$ is the infinitesimal generator of the symmetric semigroup $(P_t)$ on $L^2(\mathbb{R}^N)$. $F:=H^{1}(\mathbb{R}^N)$ is the closure of $C^{\infty}(\mathbb{R}^N)$ with respect to the norm $\|u\|_F^2:=\|u\|^2+\|\nabla u\|^2$.

Since the penalization sequence consists of the solutions defined on $\mathbb{R}^N$ without boundary conditions, we extend the functions $g$ to $\mathbb{R}^N$ by a smooth 0-extension. As the discussion in Lemma \ref{lamma. regularity of G}, we denote the function corresponding to the extended $g$ by $\bar G$, satisfying that, for any $\phi\in C^\infty_0([0,T]\times \mathbb{R}^N)$. 
\begin{eqnarray*}
\int_0^T (g(t,\cdot), \nabla \phi(t,\cdot))_{L^2(\mathbb{R}^N)}dt=\int_{0}^T\langle \bar{G}(t,\cdot),\phi(t,\cdot)\rangle_{H^{1}_0(\mathbb{R}^N)}dt.
\end{eqnarray*}
By the uniqueness of Reisz representation theorem, it is easily to know that $G=\bar G$ on $D$.
%\begin{proposition}
%For $g\in L^2([0,T]\times \mathbb{R}^N; \mathbb{R}^N)$, there is a unique function $\bar{G}\in L^2([0,T]; {\red H^{1}_0(\mathbb{R}^N)})$ such that, for any $\Phi\in C^\infty_0([0,T]\times \mathbb{R}^N)$,
%{\red\begin{eqnarray*}
%\int_0^T (g(t,\cdot), \nabla \phi(t,\cdot))_{L^2(\mathbb{R}^N)}dt=\int_{0}^T\langle \bar{G}(t,\cdot),\phi(t,\cdot)\rangle_{H^{1}_0(\mathbb{R}^N)}dt.
%\end{eqnarray*}}
%\end{proposition}
%\begin{remark}
%By the uniqueness of $G\in H^{1}(D)$ and $\bar{G}\in H^1_0(\mathbb{R}^N)$, it is follows that $\bar{G}=G$ on $D$ and $\nabla G=\nabla \bar{G}$ on $\partial D$.
%\end{remark}
%\vspace{3mm}
%{\red\begin{remark}
%Since we suppose $g=0$ on $\bar{D}^c$, it is easy to know that $G=\bar{G}$ on $\bar{D}$. By Proposition $\ref{regularity G}$, Remark $\ref{regularity of G}$ and the same proof in Lemma $\ref{lamma. regularity of G}$, it is known that $\bar{G}\in C^1([0,T])\otimes H^1_0(\mathbb{R}^N)$ and $\bar{G}(t,\cdot)$ is H\"{o}lder continuous so that it is bounded.
%\end{remark}}
%The approximation sequence $u^n(t,x)$ defined on $[0,T]\times \mathbb{R}^N$ is the solution of the following equation:

Let $u^n$ be the solution of the following penalized equation: 
\begin{equation}\label{appro.pde}
\left\{\begin{split}
&\partial_t u^n(t,x)+\frac{1}{2}\Delta u^n(t,x)-\langle n\vec\delta,\nabla u^n(t,\cdot)\rangle(x)-div g(t,x)+f^n(t,x)=0,\\
&u^n(T,x)=\Phi(x).
\end{split}\right.
\end{equation}
with $f^n(t,x)=f(t,x)-nh(t,x)\langle \vec\delta(x),\vec n(x)\rangle+2n\langle \vec\delta(x), \nabla\bar{G}(t,x)\rangle$.
It is easy to check that $u^n$ also satisfies the following equation: 
\begin{equation}\label{appro.pde1}
\left\{\begin{split}
&\partial_t u^n(t,x)+\frac{1}{2}\Delta u^n(t,x)-\langle n\vec\delta,\nabla u^n(t,\cdot)\rangle(x)-div(\nabla \bar{G}(t,x))+\bar{G}(t,x)+f^n(t,x) =0,\\
&u^n(T,x)=\Phi(x).
\end{split}\right.
\end{equation}

The coefficient $\tilde f^n$ is defined as follows:
\begin{equation*}\begin{split}
&\tilde{f}^n: \mathbb{R}^+\times\mathbb{R}^N\times F\rightarrow H,\quad \tilde{f}^n(t,x,v):=f^n(t,x,v(x))-n\langle \vec{\delta},\nabla v(x)\rangle.
\end{split}\end{equation*}
It is easy to check $\tilde{f}^n$ satisfy the following Lipschitz condition:
\begin{eqnarray*}
\|\tilde{f}^n(t,u)-\tilde{f}^n(t,v)\|\leq n \| \langle \vec\delta, \nabla (u-v)\rangle\|
\leq n\|\vec\delta\| \|u-v\|_F.
\end{eqnarray*}

By $\cite{DS04}$, the following theorem is obtained.
\begin{theorem}
There exists a unique solution $u^n$ for the following PDE
\begin{equation}\label{PDEtilde}
\left\{\begin{split}
&\partial_t u^n(t,x)+\frac{1}{2}\Delta u^n(t,x)-n\langle \vec \delta,\nabla u^n(t,x)\rangle+ f^n(t,x)-div{g}(t,x)=0,\\
&u^n(T,x)=\Phi(x).
\end{split}\right.
\end{equation}
Moreover, $u^n$ satisfies the following estimate
\begin{equation*}
\sup_{t\in[0,T]} \|u^n_t\|^2+\int_{0}^{T}\|\nabla u^n_t\|^2dt \leq\,C\left[ \|\Phi\|^2+\int_{0}^{T} \|f_t\|^2+ n^2\|\delta \|^2\|h_t\|^2+\|g_t\|^2 dt \right].
\end{equation*}
\end{theorem}
\vspace{3mm}
%As PDE \eqref{PDEtilde} admits a solution, the corresponding BSDE can be obtained by $\cite{S}$.
\begin{theorem}\label{solutionofpenalizedequation}
Let $u^n$ be the solution of PDE \eqref{appro.pde} and $\{X^n_t\}$ be the diffusion satisfying \eqref{appro.diffusion}.
(1) Set $Y^n_t=u^n(t,X^n_t)$, $Z^n_t=\nabla u^n(t,X^n_t)$.  $(Y^n,Z^n)$ solves the following BSDE
\begin{equation*}
Y^n_t=\Phi(X^n_T)-\int_{t}^{T}\langle Z^n_r,dB_r\rangle +\int_{t}^{T}f^n(r,X^n_r)dr+\int_{t}^{T}g(r,X^n_r )\ast dX^n_r.
\end{equation*}
(2) Set $\tilde{Y}^n_t=Y^n_t-2\bar{G}(t,X^n_t)$, $\tilde{Z}^n_t=Z^n_t-2\nabla \bar{G}(t,X^n_t)$, then $(\tilde Y^n,\tilde{Z}^n)$ solves
\begin{equation}\label{decomposition penalizition BSDE}
\begin{split}
\tilde{Y}^n_t=&\Phi(X^n_T)-2\bar{G}(T,X^n_T)+\int_t^T\langle \tilde{Z}^n_r, dB_r\rangle +\int_t^T f(r,X^n_r)
+\bar{G}(r,X^n_r)\,dr\\
&+2\int_t^T\partial_t\bar{G}(r,X^n_r)dr-n\int_t^T \langle \vec\delta, \vec n\rangle(X^n_r) h(r,X^n_r)dr
\end{split}\end{equation}
Furthermore,
$$\sup_{n}E\Big[\sup_{t\in[0,T]}|\tilde{Y}^n_t|^2+\int_0^T|\tilde{Z}^n_t|^2dt\Big]<+\infty.$$

\end{theorem}
\begin{proof}
	(1) is proved in Proposition 4.2 \cite{S} and then BSDE \eqref{decomposition penalizition BSDE} is estabilished by decomposition of $G(t,X^n_t)$.
Set
\begin{equation*}\begin{split}
&(1) L^n_t=-n\int_{0}^t\langle \vec\delta, \vec n\rangle(X^n_r)dr;\\
&(2) F^n(t,\cdot)=f(t,X^n_t) +\bar{G}(t,X^n_t)+2\partial_t\bar{G}(t,X^n_t);\\
&(3) H^n(t,\cdot)=h(t,X^n_t).
\end{split}\end{equation*}
Since $\langle \vec\delta, \nabla \psi\rangle\leq 0$, $\{L^n_t\}_{t\in[0,T]}$ is an increasing process. 

Applying It\^o's formula to $|\tilde{Y}^n_t|^2$, we have
\begin{equation}\label{itotildeYn}\begin{split}
|\tilde{Y}^n_t|^2+\int_t^T |\tilde{Z}^n_r|^2dr&=|\tilde{Y}^n_T|^2+2\int_t^T \tilde{Y}^n_r\langle \tilde{Z}^n_r,dB_r\rangle
+2\int_t^T \tilde{Y}^n_rF^n(r)dr+2\int_t^T \tilde{Y}^n_rH^n(r)dL^n_r\\
&\leq |\tilde{Y}^n_T|^2+2\int_t^T \tilde{Y}^n_r\langle \tilde{Z}^n_r,dB_r\rangle
+\int_t^T |\tilde{Y}^n_r|^2dr+\int_t^T |F^n(r)|^2dr\\&\quad+2\int_t^T |\tilde{Y}^n_rH^n(r)|dL^n_r.
\end{split}\end{equation}
Taking expectation in the above equation, we get 
\begin{equation}\label{itotildeYn2}
E\Big[|\tilde{Y}^n_t|^2+\int_t^T |\tilde{Z}^n_r|^2dr\Big]\leq E|\tilde{Y}^n_T|^2+E\int_t^T|F^n(r)|^2dr+2E\int_t^T |\tilde{Y}^n_rH^n(r)|dL^n_r+E\int_t^T |\tilde{Y}^n_r|^2dr
\end{equation}
%Then, with Gronwall's lemma, we obtain
%\begin{equation*}
%\begin{split}
%\sup_{t\in[0,T]}E[e^{\mu L^n_t}|\tilde{Y}^n_t|^2]\leq\cdots
%\end{split}
%\end{equation*}
Then, thanks to Gronwall's lemma, we obtain
\begin{eqnarray}\label{itotildeYn3}
E|\tilde{Y}^n_t|^2\leq CE\Big[|\tilde{Y}^n_T|^2+\int_0^T |F^n(r)|^2dr+\int_0^T |\tilde{Y}^n_rH^n(r)|dL^n_r\Big].
\end{eqnarray}
By B-D-G's inequality, combining \eqref{itotildeYn}, \eqref{itotildeYn2} and \eqref{itotildeYn3}, we have
\begin{equation*}\begin{split}
&E\Big[\sup_{t\in[0,T]}|\tilde{Y}^n_t|^2]+\int_0^T E[|\tilde{Z}^n_r|^2]dr\Big]\\
\leq&\, C E\Big[|\tilde{Y}^n_T|^2+\int_0^T|F^n(r)|^2dr+\int_0^T |\tilde{Y}^n_rH^n(r)|dL^n_r\Big]\\
\leq&\, C E\Big[|\tilde{Y}^n_T|^2+\int_0^T|F^n(r)|^2dr\Big]+\frac{1}{2}E[\sup_{t\in[0,T]}|\tilde{Y}^n_t|^2]+
\frac{C^2}{2}K E[(L^n_T)^2],
\end{split}\end{equation*}
where $C$ is a constant dependent on $\alpha,\beta, T,\mu$.  

By the boundedness of $f(t,x)$, $\bar{G}$ and $\partial_t\bar{G}$,  we get the desired uniformly boundedness.
\end{proof}
We now turn to prove $\{(\tilde{Y}^n,\tilde{Z}^n)\}_{n\geq1}$ is a Cauchy Sequence.
\begin{corollary}
\label{cauchy process}
\begin{eqnarray}
\lim_{m,n\rightarrow \infty}E\Big[\sup_{t\in[0,T]}|\tilde{Y}^n_t-\tilde{Y}^m_t|^2+\int_0^T |\tilde{Z}^n_t-\tilde{Z}^m_t|^2dt\Big]=0.
\end{eqnarray}
\end{corollary}
\begin{proof}
Since
\begin{equation*}
d(\tilde{Y}^n_t-\tilde{Y}^m_t)=\langle \tilde{Z}^n_t-\tilde{Z}^m_t,dB_t\rangle-(F^n(t)-F^m(t)dt-(H^n(t)dL^n_t-H^m(t)dL^m_t),
\end{equation*}
applying It\^o's formula to $(\tilde{Y}^n-\tilde{Y}^m)^2$, we obtain
\begin{equation*}\begin{split}
|\tilde{Y}^n_t-\tilde{Y}^m_t|^2+\int_t^T |\tilde{Z}^n_r-\tilde{Z}^m_r|^2dr&=|\tilde{Y}^n_T-\tilde{Y}^m_T|^2-2\int_t^T (\tilde{Y}^n_r-\tilde{Y}^m_r)\langle \tilde{Z}^n_r-\tilde{Z}^m_r,dB_r\rangle\\&+2\int_t^T (\tilde{Y}^n_r-\tilde{Y}^m_r)(F^n(t)-F^m(t))dr\\&+2\int_t^T (\tilde{Y}^n_r-\tilde{Y}^m_r)(H^n(r)dL^n_r-H^m(r)dL^m_r)
\end{split}\end{equation*}
Firstly, it follows that
\begin{equation*}\begin{split}
2\int_t^T (\tilde{Y}^n_r-\tilde{Y}^m_r)(F^n(t)&-F^m(t))dt
\leq\int_t^T |\tilde{Y}^n_r-\tilde{Y}^m_r|^2dr
+3\int_t^T |\tilde{X}^n_r-\tilde{X}^m_r|^2dr\\
&+\int_t^T 3|\bar{G}(r,X^n_r)-\bar{G}(r,X^m_r)|^2+12|\partial_r\bar{G}(r,X^n_r)-\partial_r\bar{G}(r,X^m_r)|^2dr.
\end{split}\end{equation*}
Furthermore,
\begin{equation*}\begin{split}
&\int_t^T (\tilde{Y}^n_r-\tilde{Y}^m_r)(H^n(r)dL^n_r-H^m(r)dL^m_r)\\
\leq&\int_t^T (\tilde{Y}^n_r-\tilde{Y}^m_r)(h(r,X^n_r)-h(r,X^m_r))dL^n_r
+\int_t^T (\tilde{Y}^n_r-\tilde{Y}^m_r) h(t,X^m_r)d(L^n_r-L^m_r)\\
\leq& \frac{1}{2}\sup_{r\in[0,T]} |\tilde{Y}^n_r-\tilde{Y}^m_r|^2+4C^2 \Big(\int_0^T |X^n_r-X^m_r|dL^n_r\Big)^2
+4K^2 (L^n_T-L^m_T)^2.
\end{split}\end{equation*}

By Gronwall's inequality and standard calculation, there is a constant $C'>0$ depending on $C,\alpha,\beta,T,K$, such that
\begin{eqnarray*}
&&E[\sup_{t\in[0,T]}|\tilde{Y}^n_t-\tilde{Y}^m_t|^2+\int_t^T |\tilde{Z}^n_r-\tilde{Z}^m_r|^2dr]\\
&\leq& C'\{ E[\int_0^T|\bar{G}(r,X^n_r)-\bar{G}(r,X^m_r)|^2+|\partial_r\bar{G}(r,X^n_r)-\partial_r\bar{G}(r,X^m_r)|^2+|X^n_r-X^m_r|^2dr]\\
  &+&E[\sup_{t\in[0,T]}|X^n_t-X^m_t|^4]^\frac{1}{2}E[L^n_T]^{\frac{1}{2}}+E[(L^n_T-L^m_T)^2]  \}\\
&\rightarrow& 0,\ \ \mbox{as}\ \ m,n\rightarrow \infty,
\end{eqnarray*}
where the limit is obtained by the H\"{o}lder continuity of $G,\partial_t G$ and uniform boundedness in Theorem \ref{solutionofpenalizedequation} and Proposition \ref{appro.diffusion estimate}.
\end{proof}

\subsection{The linear Neumann boundary problem}

In this section, we will prove the existence and uniqueness of solution for PDE \eqref{linearPDE}. \\
For fixing starting point $(t,x)\in [0,T]\times \bar{D}$, the reflecting diffusion is defined as follows
\begin{equation*}
\left\{\begin{split}
&X^{t,x}_s=x+(B_{t\vee s}-B_t)+\int_{t}^{t\vee s}\vec n (X^{t,x}_r)dL^{t,x}_r, \ for\ s\geq 0,\\
&L^{t,x}_s=\int_{t}^{t\vee s} I_{\{X^{t,x}_r\in \partial D\}}dr.
\end{split}\right.
\end{equation*}

By Theorem \ref{theorem. G}, the following BSDE
\begin{equation*}
\tilde{Y}^{t,x}_s=\Phi(X^{t,x}_T)-2G(T,X^{t,x}_T)
+\int_s^T \tilde f(r,\tilde Y^{t,x}_r)dr+\int_s^T\tilde{h}(r,\tilde{Y}^{t,x}_r)dL^{t,x}_r +\int_s^T\langle\tilde{Z}^{t,x}_r,dB_r\rangle,
\end{equation*}
has a unique solution $(\tilde{Y}^{t,x}_s,\tilde{Z}^{t,x}_s)$ for $s\in[t,T]$.

Set $\tilde{u}(t,x)=\tilde{Y}^{t,x}_t$, by $\cite{PZ}$, it is known that $\tilde{u}\in C([0,T]\times \bar{D})\subset C([0,T];L^2(D))$. Furthermore, by \cite{PP} and \cite{BM01}, 
$\{\tilde{Z}^{t,x}_s\}_{s\in[t,T]}$ has an a.s. continuous version which is given by
$$
\tilde{Z}^{t,x}_t=\nabla\tilde{Y}^{t,x}_t=\nabla \tilde{u}(t,X^{t,x}_t) \ \ \mbox{and}\ \ \tilde{Z}^{t,x}_s=\nabla \tilde{u}(s,X^{t,x}_s).
$$

%From Theorem $\ref{theorem. G}$, we know that 
%\begin{eqnarray*}
%&&\sup_{r\in[t,T]}\int_D E[|\tilde{u}(r,x)|^2]+\int_t^T \int_D E[\|\nabla \tilde{u}(r,x)\|^2]dr\\
%&=&\sup_{r\in[t,T]}\int_D E[|\tilde{u}(r,X^{t,x}_r)|^2]+\int_t^T \int_D E[\|\nabla \tilde{u}(r,X^{t,x}_r)\|^2]dr<\infty.
%\end{eqnarray*}
By the estimate in Theorem $\ref{theorem. G}$,  for every $t\in[0,T]$,
\begin{equation}\begin{split}
\int_t^T\int_D |\tilde u(r,x)|^2+|\nabla \tilde u(r,x)|^2dxdr
 &\leq C \int_t^T\int_D E[|\tilde u(r,X^{t,x}_r)|^2+|\nabla \tilde u(r,X^{t,x}_r)|^2 ]dxdr\\
&=C\int_t^T\int_D E[|\tilde Y^{t,x}_r|^2+|\tilde Z^{t,x}_r|^2 ]dxdr<+\infty,
\end{split}\end{equation}
where the first inequality is proved in \cite{BM01}. Therefore, $\tilde u\in L^2([0,T]; H^1(D))$.

The approximating process
\begin{eqnarray*}
X^{n,t,x}_s=x+(B_{t\vee s}-B_t)+\int_{t}^{t\vee s} (-n\vec\delta)(X^{n,t,x}_r)dr, \ for\ s\geq 0.
\end{eqnarray*}

Let $\{P^{n,t}_s\}_{s\geq t}$ be the semigroup and $L^{n}=\frac{1}{2}\Delta -\langle n\vec\delta, \nabla\cdot \rangle$ be the generator corresponding to $\{X^{n,t,\cdot}\}$. It is obvious that, for $f\in L^2(\mathbb{R}^N)$, $P^{n,t}_r f(x)=f(x)$, $r\in [0,t]$, and $P^{n,t}_r f(x)=E[f(X^{n,t,x}_r)]$, $r\in[t,T]$.

From the last section, we know that the solution $(\tilde{Y}^{n,t,x},\tilde{Z}^{n,t,x})$ of the following BSDE
\begin{equation}
\label{appro. BSDE}\begin{split}
\tilde{Y}^{n,t,x}_s&=(\Phi-2\bar{G})(X^{n,t,x}_T)-\int_s^T \tilde{Z}^{n,t,x}_r dB_r+\int_s^T F^n(r, \tilde{Y}^{n,t,x}_r)dr\\
&+\int_s^T H^n(r, \tilde{Y}^{n,t,x}_r)dL^{n,t,x}_r,\ \ s\in[t,T],
\end{split}\end{equation}
and $\tilde{Y}^{n,t,x}_s=\tilde{Y}^{n,t,x}_t$, $s\in[0,t]$, with $L^{n,t,x}_s=-\int_t^{t\vee s} \langle n\vec\delta, \nabla\psi\rangle(X^{n,t,x}_r)dr$, satisfies the following relationships
$$
(u^n-2\bar{G})(t,x)=(u^n-2\bar{G})(t,X^{n,t,x}_t)=\tilde{Y}^{n,t,x}_t
$$
and
$$
\tilde{Y}^{n,t,x}_s=(u^n-2\bar{G})(s, X^{n,t,x}_s),\ \mbox{for}\ s\in [t,T].
$$
Since, for $t\leq s\leq T$,
\begin{equation*}\begin{split}
P^{n,t}_s u^n_s-P^{n,t}_t u^n_t&=\int_t^s P^{n,t}_r (\partial_r +\frac{1}{2}\Delta+\langle b-n\vec\delta,\nabla\cdot\rangle)u^n_r dr\\
&=\int_t^s P^{n,t}_r (div(\nabla\bar{G})-\bar{G}-f)dr+\int_t^s P^{n,t}_r\langle n\vec\delta,\vec n h-2\nabla\bar{G}\rangle dr,
\end{split}\end{equation*}
then for $\forall\phi(t,x)\in C^\infty([0,T])\otimes C^\infty_0(\mathbb{R}^N)$, we have
\begin{equation}
\label{weak solution un}\begin{split}
(P^{n,t}_Tu^n_T,\phi_T)-(P^{n,t}_tu^n_t,\phi_t)&=\int_t^T (P^{n,t}_r (div(\nabla\bar{G})-\bar{G}-f),\phi_r)dr\\&+
\int_t^T( P^{n,t}_r\langle n\vec\delta,\vec n h-2\nabla\bar{G}\rangle,\phi_r)dr.
\end{split}\end{equation}

Specially, taking expectation on both sides of $(\ref{appro. BSDE})$ and letting $s=t$, we obtain
\begin{eqnarray*}
u^n(t,x)=P^{n,t}_T \Phi(x)+\int_t^T P^n_{r}(div(\nabla\bar{G})-\bar{G}-f)dr
+\int_t^T P^n_{r}\langle n\vec\delta,\vec n h-2\nabla\bar{G}\rangle dr,
\end{eqnarray*}
which means $u^n$ is also a mild solution of PDE $(\ref{appro.pde})$.

By Corollary $\ref{cauchy process}$, we deduce
\begin{equation*}
\sup_{r\in[t,T]}\int_D E[|\tilde{Y}^{n,t,x}_r-\tilde{Y}^{m,t,x}_r|^2]dx+\int_t^T\int_D E |\tilde{Z}^{n,t,x}_r-\tilde{Z}^{m,t,x}_r|^2]dxdr
\rightarrow 0,\end{equation*}
as $m,n\rightarrow \infty, \forall t\in[0,T]$.
\begin{proposition}
\label{proposition.limit}
For $(t,x)\in[0,T]\times D$,
\begin{equation}
\label{limit process}
\lim_{n\rightarrow \infty}\int_D E[\sup_{r\in[t,T]}|\tilde{Y}^{t,x}_r-\tilde{Y}^{n,t,x}_r|^2]dx
+\int_t^T\int_D E[|\tilde{Z}^{t,x}_r-\tilde{Z}^{n,t,x}_r|^2|]dxdr=0.
\end{equation}
Furthermore, for every $t\in[0,T]$,

(1) $u^n(t,x)$ is a Cauchy sequence in $L^2(D)$ and
$$
\bar{u}(t, x):=L^2-\lim_{n\rightarrow \infty} u^n(t,x);
$$

(2) $\bar{u}(t, x)=\tilde u(t,x)+2G(t,x)$;

(3) for $r\in [t,T]$, $\lim \limits_{n\rightarrow\infty}P^{n,t}_ru^n_r=P^t_r\bar{u}_r$.

%\begin{enumerate}
%\item $u^n(t,x)$ is a Cauchy sequence in $L^2(D)$, and the limit $\bar{u}_t$ has an a.s. continuous version $\tilde{u}(t,x)+2G(t,x)$.
%\item For $r\in [t,T]$, $\lim \limits_{n\rightarrow\infty}P^{n,t}_ru^n_r=P^t_r\bar{u}_r$.
%\end{enumerate}
\end{proposition}

\begin{proof}
Since $G=\bar G$ on $D$, and $\bar{G}$ is Lipschitz continuous,
\begin{equation*}\begin{split}
\lim_{n\rightarrow \infty}E\int_t^T|\bar{G}(r, X^{n,t,x}_r)-G(r,X^{t,x}_r)|^2 dr&\leq \sup_{x\in D} E\int_t^T |X^{n,t,x}_r-X^{t,x}_r|^2 dr\\
&\leq T\sup_{x\in D}E[\sup_{r\in[0,T]}|X^{n,t,x}_r-X^{t,x}_r|^2]=0.
\end{split}\end{equation*}
Similarly, by the standard calculus in Corollary $\ref{cauchy process}$, we  find that $(\tilde{Y}^{n,t,x},\tilde{Z}^{n,t,x})$ is a Cauchy sequence and the limit is $(\tilde{Y}^{t,x},\tilde{Z}^{t,x})$ which is shown in $(\ref{limit process})$.\\
Furthermore, for fixing $t\in[0,T]$, 
\begin{equation*}\begin{split}
\int_D |u^n(t,x)-u^m(t,x)|^2dx&=\int_{D}|(u^n-2\bar{G})(t,X^{n,t,x}_t)-(u^m-2\bar{G})(t,X^{m,t,x}_t)|^2dx\\
&=\int_D E|\tilde{Y}^{n,t,x}_t-\tilde{Y}^{m,t,x}_t|^2dx\\
&\leq \sup_{r\in[t,T]}\int_D E |\tilde{Y}^{n,t,x}_r-\tilde{Y}^{m,t,x}_r|^2dx\rightarrow 0,\ \mbox{as}\ m,n\rightarrow \infty.
\end{split}\end{equation*}
This shows that $(u^n)_n$ is a Cauchy sequence and we denote the limit as $\bar{u}$.  

On the other hand, \eqref{limit process} implies   
\begin{equation*}\begin{split}
\int_D|\tilde u(t,x)+2G(t,x)-u^n(t,x)|^2dx&=E\int_D |\tilde{Y}^{t,x}_t-\tilde Y^{n,t,x}_t|^2dx\\
&\leq\int_D E[\sup_{r\in[t,T]}|\tilde{Y}^{t,x}_r-\tilde{Y}^{n,t,x}_r|^2]dx\rightarrow 0,\mbox{as}\ \ n\rightarrow\infty.
\end{split}\end{equation*}
By the uniqueness of limit of Cauchy sequence, $\bar{u}(t,x)=\tilde{Y}^{t,x}_t+2G(t,x)$ for every $t\in[0,T]$. Then (1) and (2) are proved. \\
The H\"{o}lder continuity of $\bar{G}$ provides
\begin{equation*}\begin{split}
&\int_D |P^{n,t}_r u^n_r(x)-P^{t}_r \bar{u}(x)|^2 dx
\leq \int_D E[|u^n(r,X^{n,t,x}_r)-\bar{u}(r,X^{t,x}_r)|^2]dx\\
&\leq 2\int_D E[|\tilde{Y}^{n,t,x}_r-\tilde{Y}^{t,x}_r|^2]dx+8\int_D E[|\bar{G}(r,X^{n,t,r})-\bar{G}(r,X^{t,x}_r)|^2]dx
\rightarrow 0,\ \ \mbox{as} \ \ n\rightarrow \infty.
\end{split}\end{equation*}
The third conclusion is obtained.
\end{proof}
%\textcolor[rgb]{1.00,0.00,0.00}{Let's remark that, there exists a Borel measurable function $\upsilon(t,x):[0,T]\times D\rightarrow \mathbb{R}^N\times \mathbb{R}^N$ such that
%$$
%\tilde{Z}^{t,x}_r=\upsilon(r,X^{t,x}_r).
%$$}
\begin{theorem}\label{theorem linear}
$\bar{u}\in C([0,T];L^2(D))\cap L^2([0,T];H^1(D))$ is the unique weak solution of the Neumann boundary problem \eqref{linearPDE}.
\end{theorem}
\begin{proof}
\textbf{Existence:} 
Taking limit on both sides of \eqref{weak solution un}, by Proposition $\ref{proposition.limit}$, we have
\begin{equation*}\begin{split}
(P^{t}_T\Phi,\phi_T)-(P^{t}_s\bar{u}_s,\phi_s)=&\int_s^T (P^{t}_r (div(\nabla\bar{G})-\bar{G}-f),\phi_r)dr\\&+\lim_{n\rightarrow \infty}\int_s^T( P^{n}_r\langle n\vec\delta,\vec n h-2\nabla\bar{G}\rangle,\phi_r) dr.
\end{split}\end{equation*}
Furthermore, 
\begin{equation*}\begin{split}
&\lim_{n\rightarrow \infty}\int_t^T( P^{n,t}_r\langle n\vec\delta,\vec n h-2\nabla\bar{G}\rangle,\phi_r) dr\\
=&\,\lim_{n\rightarrow \infty} \int_D \int_t^T E[\phi_r(X^{n,t,x}_t)\langle (\vec n h-2\nabla\bar{G})(X^{n,t,x}_r),dK^{n,x,r}\rangle]\\
=&\,\frac{1}{2} \int_D \int_t^T E[\phi_r(X^{t,x}_t)( h- 2\langle\nabla\bar{G},\vec{n}\rangle)(X^{t,x}_r)dL^{x}_r].
\end{split}\end{equation*}
Therefore,
\begin{equation*}
P^{t}_s\bar{u}_s=P^{t}_T\Phi-\int_s^T P^{t}_r (div(\nabla\bar{G})-\bar{G}-f)dr
-\frac{1}{2} E\int_s^T( h- 2\langle\nabla\bar{G},\vec{n}\rangle)(X^{t,x}_r)dL^{t,x}_r.
\end{equation*}
This implies $\bar{u}$ is a mild solution, and we will prove that it is also a weak solution.

Firstly, we know that, for $v\in H^1(D)$, $t\rightarrow(P^{t}_s \phi_s,v)$ is differentiable on $[0,s]$, and
\begin{equation*}
\frac{\partial}{\partial t }(P^{t}_s \phi_s,v)=(-LP^{t}_s \phi_s,v)=\frac{1}{2}(\nabla P^{t}_s \phi_s,\nabla v).
\end{equation*}
Then it follows that
\begin{equation*}\begin{split}
\frac{\partial}{\partial t}(\bar{u}_t,v)
=&\frac{\partial}{\partial t}(P^{t}_T\Phi,v)+(div(\nabla\bar{G})-\bar{G}-f,v)
-\int_t^T \frac{\partial }{\partial t}(P^{t}_r (div(\nabla\bar{G})-\bar{G}-f),v)dr\\
&+\int_{\partial D}( h-2\langle \bar{G},\vec{n}\rangle)(x) d\sigma(x)\\
=&-L(P^{t}_T\Phi,v)-\int_t^T L(P^{t}_r (div(\nabla\bar{G})-\bar{G}-f),v)dr+(div(\nabla\bar{G})-\bar{G}-f,v)\\
&+\int_{\partial D}( h-2\langle \bar{G},\vec{n}\rangle)(x) d\sigma(x)\\
=&(-L \bar u_t,v)+(div(\nabla\bar{G})-\bar{G}-f,v)+\int_{\partial D}( h-2\langle \bar{G},\vec{n}\rangle)(x) d\sigma(x),
\end{split}\end{equation*}
where $(-L u(t,\cdot),v)=\frac{1}{2}(\nabla u(t,\cdot),\nabla v)$.\\

\textbf{Uniqueness:}  Suppose $\hat u\in  C([0,T];L^2(D))\cap L^2([0,T];H^1(D))$ is another solution for PDE \eqref{linearPDE}.
Set $\hat{Y}^{t,x}_s=\hat u(s,X^{t,x}_s)$ and 
$\hat{Z}^{t,x}_s=\nabla \hat u(s,X^{t,x}_s)$,  by Proposition \ref{decompositionofu}, 
$(\hat{Y}^{t,x}_s, \hat{Z}^{t,x}_s)$ admits the representation \eqref{decompositionofu}. 
Therefore, $(\hat{Y}^{t,x}_s-2 G(X^{t,x}_s), \hat{Z}^{t,x}_s-2\nabla G(X^{t,x}_s))$ satisfies \eqref{BSDE-1}. 
By the uniqueness of solution for \eqref{BSDE-1}, we know that 
\begin{equation*}\begin{split}
E\Big[\sup_{s\in[t,T]}|(\hat u-2G)&(s,X^{t,x}_s)-(\bar u-2G)(s,X^{t,x}_s)|^2\\
&+\int_t^T |\nabla(\hat u-2G)(s,X^{t,x}_s)-\nabla (\bar u-2G)(s,X^{t,x}_s)|^2ds\Big]=0,
\end{split}\end{equation*}
which provides that $\hat u=\bar u$ and $\nabla\hat u=\nabla \bar u$.
\end{proof}

\section{Existence and uniqueness of solution for nonlinear Neumann boundary problem}

Now we will prove the result in the nonlinear case by Picard iteration.
Let us consider the Picard sequence $(u^n)_n$ defined by  $u^0=0$ and for all $n\in \mathbb{N}^*$ we denote by $u^{n+1}$ the solution of the linear PDE: 
\begin{equation}\label{picard eq.}
\left\{
\begin{split}
&\partial_t u^{n+1}(t,x)+ \frac{1}{2}\Delta u^{n+1}(t,x)-divg(t,x,u^n,\nabla u^n)+f(t,x,u^{n},\nabla u^n)=0,\ \mbox{on}\ [0,T]\times D,\\
&u^{n+1}(T,x)=\Phi(x),\ \mbox{on}\ D,\\
&\langle \nabla u^{n+1}(t,x)-2g(t,x,u^n,\nabla u^n), \vec{n}(x)\rangle=h(t,x,u^{n}),\ \mbox{on}\ [0,T]\times \partial D.
\end{split}
\right.
\end{equation}
By the result in last section, we know there exists a  unique solution of linear PDE \eqref{picard eq.} for every $n\in \mathbb{N}$.  In the following discussion we will prove the convergence of $\{u^n\}$ in both analytic and probabilistic method independently. 

\subsection{Analytic Method}
\begin{theorem}
Suppose \textbf{(H1)-(H3)}, \textbf{(H6)} hold, then PDE 
\begin{equation}\label{PDE new}
\left\{ \begin{split}
&\partial_t u(t,x)+ \frac{1}{2}\Delta u(t,x)-divg(t,x,u,\nabla u)+f(t,x,u,\nabla u)=0, (t,x)\in [0,T]\times D\\
&u(T,x)=\Phi(x), \ x\in D,\\
&\frac{\partial u}{\partial\vec{n}}(t,x)-2\langle g(t,x,u,\nabla u), \vec{n}\rangle+h(t,x,u)=0, \ (t,x)\in \ [0,T]\times \partial D,
\end{split}
\right.
\end{equation}

 has a unique  weak solution.
\end{theorem}
\begin{proof}
For simplicity, in this section, we set $g^n(t,x)=g(t,x,u^n,\nabla u^n)$, $h^n(t,x)=h(t,x,u^n)$ and $f^n(t,x)=f(t,x,u^n,\nabla u^n)$. Choosing $\theta>0$, we have
\begin{equation*}\begin{split}
&\|u^{n+1}_0-u^{n}_0\|^2+\int_0^T e^{\theta s}\|\nabla(u^{n+1}_s-u^{n}_s)\|^2ds=-\theta \int_0^T e^{\theta s}\|u^{n+1}_s-u^{n}_s\|^2ds
\\&+2\int_0^Te^{\theta s}(f^n_s-f^{n-1}_s, u^{n+1}_s-u^{n}_s)ds
+2\int_0^T \int_D e^{\theta s}\langle g_s^n-g_s^{n-1},\nabla(u^{n+1}_s-u^{n}_s)\rangle dxds\\
&+2\int_0^T \int_{\partial D}e^{\theta s}(h_s^n-h_s^{n-1})(u^{n+1}_s-u^{n}_s)d\sigma(x)ds,
\end{split}
\end{equation*}
where $d\sigma$ is the $R^{N-1}$-dimensional Lebesgue measure on $\partial D$.\\
Using Cauchy-Schwarz's inequality and Lipschitz conditions, we have 
\begin{equation*}\begin{split}
&2\int_0^T \int_D e^{\theta s}\langle g_s^n-g_s^{n-1},\nabla(u^{n+1}_s-u^{n}_s)\rangle dxds\\
\leq&\, 2\int_0^Te^{\theta s}\|g_s^n-g_s^{n-1}\|\ \|\nabla(u^{n+1}_s-u^{n}_s)\| ds\\
\leq&\,  \gamma \epsilon \int_0^T e^{\theta s}\|\nabla(u^{n+1}_s-u^{n}_s)\|^2 ds
+\frac{\gamma}{\epsilon}\int_0^T e^{\theta s}\|u^{n}_s-u^{n-1}_s\|^2 _{F}\,ds
\end{split}\end{equation*}
and
\begin{equation*}\begin{split}
&2\int_0^Te^{\theta s}(f^n_s-f^{n-1}_s, u^{n+1}_s-u^{n}_s)ds\\
\leq&\, \frac{\alpha}{\epsilon}\int_0^T e^{\theta s}\|u^{n+1}_s-u^{n}_s\|^2ds+\alpha\epsilon\int_0^T e^{\theta s}\|u^{n}_s-u^{n-1}_s\|^2_{F}\, ds
\end{split}\end{equation*}
and
\begin{equation*}\begin{split}
&\int_0^T \int_{\partial D} e^{\theta s}(h_s^{n}-h_s^{n-1})(u^{n+1}_s-u^{n}_s)d\sigma(x)ds\\
\leq&\, \beta\|Tr\|^2 \epsilon_1 \int_0^T e^{\theta s}\|u^{n+1}_s-u^{n}_s\|_{F}^2\,ds+\frac{\beta\|Tr\|^2}{\epsilon_1} \int_0^T e^{\theta s}\|u^{n}_s-u^{n-1}_s\|_{F}^2\,ds,
\end{split}\end{equation*}
where $Tr: H^1(D)\rightarrow L^2(\partial D)$ is the trace operator and $\|Tr\|$ is the norm of the operator satisfying $\|v\|_{L^{2}(\partial D)}\leq \|Tr\|\|v\|_{H^1}$.
Therefore, it follows that
\begin{equation*}\begin{split}
&(\theta-\frac{\alpha}{\epsilon}-\beta\|Tr\|^2 \epsilon_1)\int_0^T e^{\theta s}\|u^{n+1}_s-u^{n}_s\|^2ds
+(1-\gamma\epsilon-\beta\|Tr\|^2 \epsilon_1)\int_0^T e^{\theta s}\|\nabla (u^{n+1}_s-u^{n}_s)\|^2ds\\
\leq& (\alpha\epsilon+\frac{\gamma}{\epsilon}+\frac{\beta\|Tr\|^2}{\epsilon_1})\int_0^T e^{\theta s}\|u^{n}_s-u^{n-1}_s\|^2ds
+(\alpha\epsilon+\frac{\gamma}{\epsilon}+ \frac{\beta\|Tr\|^2}{\epsilon_1})\int_0^T e^{\theta s}\|\nabla(u^{n}_s-u^{n-1}_s)\|^2ds.
\end{split}\end{equation*}
Choose $\epsilon$, $\epsilon_1$ such that
$$
\alpha\epsilon+\frac{\gamma}{\epsilon}+\frac{\beta\|Tr\|^2}{\epsilon_1}< 1-\gamma\epsilon-\beta\|Tr\|^2 \epsilon_1
$$
and $\theta>0$ such that
$$\frac{\theta-\frac{\alpha}{\epsilon}-\beta\|Tr\|^2 \epsilon_1}{1-\gamma\epsilon-\beta\|Tr\|^2 \epsilon_1}=1\,.$$
By setting $\rho=\frac{\alpha\epsilon+\frac{\gamma}{\epsilon}+\frac{\beta\|Tr\|^2}{\epsilon_1}}{1-\gamma\epsilon-\beta\|Tr\|^2 \epsilon_1}$, we find
\begin{equation*}
\int_0^T e^{\theta s} (\|u^{n+1}_s-u^{n}_s\|^2+\|\nabla (u^{n+1}_s-u^{n}_s)\|^2)ds
\leq \rho \int_0^T e^{\theta s} (\|u^{n}_s-u^{n-1}_s\|^2+\|\nabla (u^{n}_s-u^{n-1}_s)\|^2)ds.
\end{equation*}
Note that for fixed positive number $\theta$, the norm is defined as
$$
\|v\|^2_{\theta}:=\int_0^T e^{\theta s}( \|v_s\|^2+\|\nabla v_s\|^2)ds
$$
for $v\in L^2([0,T]; H^1(D))$ is equivalent as $\|v\|^2:=\int_0^T \|v\|^2_{H^1}ds$.

Since $\rho<1$, it follows that
\begin{eqnarray*}
\|u^{n+1}-u^{n}\|^2_{\theta,\delta}\leq \rho \|u^{n}-u^{n-1}\|^2_{\theta,\delta}\leq  \rho^2 \|u^{n-1}-u^{n-2}\|^2_{\theta,\delta}\leq\cdots
\leq  \rho^n \|u^{1}\|^2_{\theta,\delta}\rightarrow 0,\ n\rightarrow\infty.
\end{eqnarray*}
This means $(u^n)_n$ is a Cauchy sequence in $ L^2([0,T]; H^1(D))$, and its limit is denoted by $u$.

For any test function $\phi$, we have
\begin{equation*}\begin{split}
(u^n_T,\phi_T)&-(u^n_0,\phi_0)-\int_0^T(u_t^n,\partial_t\phi_t)dt=\frac{1}{2}\int_{0}^{T}\mathcal{E}(u^n_t,\phi_t)dt
-\int_{0}^{T}(f_t^n,\phi_t)dt\\&-\int_{0}^{T}\int_D\langle g^n_t,\nabla \phi_t\rangle(x)dxdt+\int_{0}^{T}\int_{\partial D} h^n_t(x)\phi_t(x)d\sigma(x)dt.
\end{split}\end{equation*}
Taking limits on both sides of the above equation, we obtain
\begin{equation*}\begin{split}
(u_T,\phi_T)&-(u_0,\phi_0)-\int_0^T(u_t,\partial_t\phi_t)dt=\frac{1}{2}\int_{0}^{T}\mathcal{E}(u_t,\phi_t)dt
-\int_{0}^{T}(f_t,\phi_t)dt\\&-\int_{0}^{T}\int_D\langle g_t,\nabla \phi_t\rangle(x)dxdt+\int_{0}^{T}\int_{\partial D} h_t(x)\phi_t(x)d\sigma(x)dt.
\end{split}\end{equation*}
which means $u$ is the weak solution of PDE \eqref{PDE}.
\end{proof}
Uniqueness:  Suppose $u,\bar u\in L^2([0,T];H^1(D))$ are two solutions for PDE \eqref{PDE new}, we obtain
\begin{equation*}\begin{split}
&\|u_0-\bar u_0\|^2+\int_0^T e^{\theta s}\|\nabla(u_s-\bar u_s)\|^2ds\\
=&-\theta \int_0^T e^{\theta s}\|u_s-\bar u_s\|^2ds
+2\int_0^Te^{\theta s}(f(s,x,u_s,\nabla u_s)-f(s,x,\bar u_s,\nabla \bar u_s), u_s-\bar u_s)ds\\
&+2\int_0^T \int_D e^{\theta s}\langle g(s,x,u_s,\nabla u_s)-g(s,x,\bar u_s,\nabla \bar u_s),\nabla(u_s-\bar u_s)\rangle dxds\\
&+2\int_0^T \int_{\partial D}e^{\theta s}(h_s(x,u)-h_s(x,\bar u))(u_s-\bar u_s)d\sigma(x)ds.
\end{split}
\end{equation*}
By the same method in the proof of existence,  there is a positive constant $\rho<1$,  such that
$$
\|u-\bar u\|^2_{\theta, \delta}\leq \rho \|u-\bar u\|^2_{\theta, \delta},
$$
which implies $\|u-\bar u\|_{\theta, \delta}=0$. Hence $u=\bar u$.
\subsection{Probabilistic Method}

Let $m$ denote the Lebesgue measure on $D$ and set the pobability space $\Omega'=D\otimes \Omega$  and probability $P^m=m\otimes P$.
%Assume $u^n$ is the solution of PDE $(\ref{appro.PDE})$  
$\{X_t\}$ is the reflecting Brownian motion in domain  D
$$
X_t-X_s=B_t-B_s+\int_s^t \vec{n}(X_r)dL_r.
$$
It is known that, $\{X_t\}$ is a symmetric diffusion with initial distribution $m$.

By the symmetricalness, we know that
$$
\bar{B}(s,t)=2X_s-2X_t+B_t-B_s=B_s-B_t-2\int_s^t \vec{n}(X_r)dL_r,
$$
is a backward martingale under $P^m$ w.r.t. the backward filtration $\mathcal F'_s=\sigma\{X_r|r\in[s,\infty)\}$.

For $g=(g_1,\cdots,g_N): \mathbb{R}^N\rightarrow \mathbb{R}^N$, as in Section 3 we define the backward stochastic integral as follows
\begin{eqnarray}\label{backwardsibarB}
\int_s^t g_i(X_{r})d\bar{B}^i_t=(L^2-)\lim_{\delta\rightarrow 0}\sum_{j=0}^{n-1} g(X_{t_{j+1}})\bar{B}^i(t_j,t_{j+1}),
\end{eqnarray}
where the limit is over the partition $s=t_0<t_1<\cdots<t_n=t$ and $\delta=\max_j (t_{j+1}-t_j)$.

In this case, one has
\begin{eqnarray}\label{siofgwrtX}
\int_s^t g\ast dX_r=\int_s^t \langle g(X_r), dB_r\rangle+\int_s^t \langle g(X_r),d\bar{B}_r\rangle+2\int_s^t \langle g,\vec n\rangle(X_r)dL_r.
\end{eqnarray}

\begin{theorem}
Suppose \textbf{(H2)-(H6)} hold, then PDE \eqref{PDE new} has a unique  weak solution.
\end{theorem}
\begin{proof}
	Existence: 
Consider the Picard iteration sequence \eqref{picard eq.} and set $M_s^{n,t,x}=u^n(s,X^{t,x}_s)$ and $N_s^{n,t,x}=\nabla u^n(s,X^{t,x}_s)$. For simplicity, we denote $M^{n,t,x}_s,N^{n,t,x}_s, X^{t,x}_s $ as $M^n_s, N^n_s, X^x_s$ respectively.  It is known that $(M^n,N^n)$ satisfies the  following BSDE,
\begin{equation*}\begin{split}
M^n_s=&\Phi(X_T)-\int_s^T \langle N^n_r,dB_r\rangle+\int_s^T f(r,X_r,M^{n-1}_r,N^{n-1}_r)dr\\
&+\int_s^T h(r,X_r,M^{n-1}_r)dL_r+\int_s^T \langle g(r,X_r,M^{n-1}_r,N^{n-1}_r, N^{n-1}_r),dB_r+d\bar B_r\rangle.
\end{split}\end{equation*}
By Ito's formula (\cite{S}), it follows that
\begin{equation*}
\begin{split}
&e^{\lambda s+\mu L_s}|M^{n+1}_s-M^n_s|^2+\int_s^T e^{\lambda r+\mu L_r} |N^{n+1}_r-N^n_r|^2dr\\
=&-2\int_s^T e^{\lambda r+\mu L_r}(M^{n+1}_r-M^n_r)\langle N^{n+1}_r-N^n_r,dB_r\rangle\\
&+2\int_s^T e^{\lambda r+\mu L_r} (M^{n+1}_r-M^n_r)(f(r,X_r,M^{n}_r,N^{n}_r)-f(r,X_r,M^{n-1}_r,N^{n-1}_r))dr\\
&+2\int_s^T e^{\lambda r+\mu L_r} (M^{n+1}_r-M^n_r)(h(r,X_r,M^{n}_r)-h(r,X_r,M^{n-1}_r))dL_r\\
&+2\int_s^T  e^{\lambda r+\mu L_r}(M^{n+1}_r-M^n_r)\langle g(r,X_r,M^{n}_r,N^{n}_r)-g(r,X_r,M^{n-1}_r,N^{n-1}_r),dB_r+d\bar B_r\rangle\\
&+2\int_s^T e^{\lambda r+\mu L_r} \langle g(r,X_r,M^{n}_r,N^{n}_r)-g(r,X_r,M^{n-1}_r,N^{n-1}_r), N^{n+1}_r-N^n_r\rangle dr\\
&-\int_s^T  e^{\lambda r+\mu L_r} |M^{n+1}_r-M^n_r|^2(\lambda dr+\mu L_r).
\end{split}
\end{equation*}
By a standard calculation, we obtain
\begin{equation*}
\begin{split}
E_m\Big[&(\lambda-\epsilon_1)\int_s^Te^{\lambda r+\mu L_r} |M^{n+1}_r-M^n_r|^2dr
+(1-\epsilon_3)\int_s^Te^{\lambda r+\mu L_r} |N^{n+1}_r-N^n_r|^2dr\\
&+(\mu-\epsilon_2) \int_s^Te^{\lambda r+\mu L_r} |M^{n+1}_r-M^n_r|^2dL_r\Big]\\
\leq E_m\Big[&(\frac{\alpha^2}{\epsilon_1}+\frac{\gamma^2}{\epsilon_3})\int_s^Te^{\lambda r+\mu L_r} (|M^{n}_r-M^{n-1}_r|^2+|N^{n}_r-N^{n-1}_r|^2)dr\\
&+\frac{\beta^2}{\epsilon_3}\int_s^Te^{\lambda r+\mu L_r} |M^{n}_r-M^{n-1}_r|^2dL_r\Big]
\end{split}
\end{equation*}

Since $2\sqrt{2}\gamma<1$, we choose $\epsilon_1,\epsilon_3$ such that
$$
\frac{\alpha^2}{\epsilon_1}+\frac{\gamma^2}{\epsilon_3}<1-\epsilon_3,
$$
then chose $\lambda$ such that
$$
\lambda-\epsilon_1=1-\epsilon_3,
$$
finally chose $\mu$ such that
$$
\frac{\frac{\beta^2}{\epsilon_3}}{\frac{\alpha^2}{\epsilon_1}+\frac{\gamma^2}{\epsilon_3}}
=\frac{\mu-\epsilon_2}{1-\epsilon_3}.
$$
Let $\rho=\frac{\frac{\alpha^2}{\epsilon_1}+\frac{\gamma^2}{\epsilon_3}}{1-\epsilon_3}$ and $\delta=\frac{\mu-\epsilon_2}{1-\epsilon_3}$. We obtain
\begin{equation*}\begin{split}
&E_m\Big[\int_s^Te^{\lambda r+\mu L_r} \left((|M^{n+1}_r-M^n_r|^2+|N^{n+1}_r-N^n_r|^2)dr+\delta |M^{n+1}_r-M^n_r|^2dL_r\right)\Big]\\
\leq&\, \rho E_m\Big[\int_s^Te^{\lambda r+\mu L_r} \left((|M^{n}_r-M^{n-1}_r|^2+|N^{n}_r-N^{n-1}_r|^2)dr+\delta |M^{n}_r-M^{n-1}_r|^2dL_r\right)\Big]\\
\leq&\,\cdots\\
\leq&\, \rho^nE_m\Big[\int_s^Te^{\lambda r+\mu L_r} \left((|M^{1}_r|^2+|N^{1}_1|^2)dr+\delta |M^{1}_r|^2dL_r\right)\Big]\rightarrow 0,\ \ \ n\rightarrow\infty.
\end{split}\end{equation*}
Therefore, $(e^{\lambda \cdot+\mu L_\cdot}M^n, e^{\lambda \cdot+\mu L_\cdot}N^n)$ is a Cauchy sequence in $L^2([t,T]\times D)\otimes L^2([t,T]\times D)$ and the limit is denoted by $(\tilde M^n, \tilde N^n)$.

Set $$M_t=e^{-\lambda t-\mu L_t}\tilde M_t \quad \mbox{and} \quad N_t=e^{-\lambda t-\mu L_t}\tilde N_t\,.$$
It is easy to check that $(M,N)$ satisfies the following BSDE:
\begin{equation*}\begin{split}
	M_s=&\,\Phi(X_T)-\int_s^T \langle N_r,dB_r\rangle+\int_s^T f(r,X_r,M_r,N_r)dr\\
	&+\int_s^T h(r,X_r,M_r,N_r)dL_r+\int_s^T \langle g(r,X_r,M_r,N_r), dB_r+d\bar B_r\rangle.
\end{split}\end{equation*}
Set $u_0(t,x)=E^x[M_t]$, $v_0(t,x)=E^x[N_t]$ and then $u_0(t,\cdot),v_0(t,\cdot)\in L^2(D)$. By the Theorem \ref{theorem linear}, the following equation has unique solution $v\in C([0,T]; L^2(D))\cap L^2([0,T]; H^1(D))$,
\begin{equation}
\left\{ \begin{split}
&\partial_t v(t,x)+ \frac{1}{2}\Delta v(t,x)-divg(t,x,u_0,v_0)+f(t,x,u_0,v_0)=0, \ (t,x)\in [0,T]\times D,\\
&v(T,x)=\Phi(x), \quad x\in D,\\
&\frac{\partial v}{\partial\vec{n}}(t,x)- 2\langle g(t,x,u_0,v_0), \vec{n}(x)\rangle=h(t,x,u_0), \quad (t,x)\in [0,T]\times \partial D.
\end{split}
\right.
\end{equation}
Set $\tilde M_s=v(X_s)$ and $\tilde N_s=\nabla v(X_s)$, by Proposition \ref{decompositionofu}, $(\tilde M, \tilde N)$ solves the following BSDE
\begin{equation*}
\begin{split}
\tilde M_s=&\,\Phi(X_T)-\int_s^T \langle \tilde N_r,dB_r\rangle+\int_s^T f(r,X_r,M_r,N_r)dr\\
&+\int_s^T h(r,X_r,M_r,N_r)dL_r+\int_s^T \langle g(r,X_r,M_r,N_r),dB_r+d\bar B_r\rangle.
\end{split}
\end{equation*}
Since 
$$
M_s-\tilde M_s=-\int_s^T\langle N_r-\tilde N_r,dB_r\rangle,
$$
taking conditional expectation on both sides of the above equality, 
$$
M_s-\tilde M_s=-E\Big[\int_s^T\langle N_r-\tilde N_r,dB_r\rangle|\mathcal{F}_s\Big]=0.
$$
Furthermore, since
$$
|M_s-\tilde M_s|^2=-\int_s^T (M_r-\tilde M_r)\langle N_r-\tilde N_r,dB_r\rangle-\int_s^T |N_r-\tilde N_r|^2dr,
$$
we obtain $E\int_s^T |N_r-\tilde N_r|^2dr=0$ which deduces that $N_r=\tilde N_r$ for $r\in[t,T]$. 

Therefore, $v=u_0$ and $\nabla v=v_0$, which implies $v$ is a solution for nonlinear PDE \eqref{PDE}.

Uniqueness:  Suppose $u,v$ are two solutions for \eqref{PDE new}.  Set $Y_t=u(t,X_t), Z_t=\nabla u(t,X_t)$ and $\tilde Y_t=v(t,X_t), \tilde Z_t=\nabla v(t,X_t)$.  It follows that
\begin{equation*}
\begin{split}
&e^{\lambda s+\mu L_s}|Y_s-\tilde Y_s|^2+\int_s^T e^{\lambda r+\mu L_r} |Z_r-\tilde Z_r|^2dr\\
=&-2\int_s^T e^{\lambda r+\mu L_r}(Y_r-\tilde Y_r)\langle Z_r-\tilde Z_r, dB_r\rangle\\
&+2\int_s^T e^{\lambda r+\mu L_r} (Y_r-\tilde Y_r)(f(r,X_r, Y_r,Z_r)-f(r,X_r,\tilde Y_r,\tilde Z_r))dr\\
&+2\int_s^T e^{\lambda r+\mu L_r} (Y_r-\tilde Y_r)(h(r,X_r, Y_r)-h(r,X_r,\tilde Y_r))dL_r\\
&+2\int_s^T  e^{\lambda r+\mu L_r}(Y_r-\tilde Y_r)\langle g(r,X_r, Y_r,Z_r)-g(r,X_r,\tilde Y_r,\tilde Z_r), dB_r+d\bar B_r\rangle\\
&+2\int_s^T e^{\lambda r+\mu L_r} \langle g(r,X_r, Y_r,Z_r)-g(r,X_r,\tilde Y_r,\tilde Z_r),Z_r-\tilde Z_r \rangle dr\\
&-\int_s^T  e^{\lambda r+\mu L_r} |Y_s-\tilde Y_s|^2(\lambda dr+\mu L_r).
\end{split}
\end{equation*}
By the same calculation in the proof of existence, we find a positive number $\rho<1$, such that
\begin{eqnarray*}
	&&E_m\Big[\int_0^Te^{\lambda r+\mu L_r} \left((|Y_r-\tilde Y_r|^2+|Z_r-\tilde Z_r|^2)dr+\delta |Y_r-\tilde Y_r|^2dL_r\right)\Big]\\
	&\leq& \rho \ E_m\Big[\int_0^Te^{\lambda r+\mu L_r} \left((|Y_r-\tilde Y_r|^2+|Z_r-\tilde Z_r|^2)dr+\delta |Y_r-\tilde Y_r|^2dL_r\right)\Big],
\end{eqnarray*}
which implies that $Y_t=\tilde Y_t, Z_r=\tilde Z_r$. Hence $u=v$ and $\nabla u=\nabla v$.
\end{proof}

%Recall Theorem 1, the following BSDE
%\begin{equation*}
%\tilde{Y}^{t,x}_s=\Phi(X^{t,x}_T)-2G(T,X^{t,x}_T)
%+\int_s^T \tilde f(r,\tilde Y^{t,x}_r,\tilde{Z}^{t,x}_r)dr+\int_s^T\tilde{h}(r,\tilde{Y}^{t,x}_r)dL^{t,x}_r +\int_s^T\langle\tilde{Z}^{t,x}_r,dB_r\rangle,
%\end{equation*}
%has a unique solution $(\tilde{Y}^{t,x}_s,\tilde{Z}^{t,x}_s)$ for $s\in[t,T]$.
%
%
%It is known that  $\tilde{u}(t,x)=\tilde{Y}^{t,x}_t$ and $\nabla \tilde{u}(t,x)=\tilde{Z}^{t,x}_t$. 
%
%By the Theorem \ref{theorem linear}, the following PDE has a unique solution:
%
%\begin{equation}
%\label{PDE}
%\left\{ \begin{split}
%&\partial_t v+ \frac{1}{2}\Delta v-div(g(t,\cdot,\tilde u+2G,\nabla (\tilde u+2G)))+f(t,\cdot,\tilde u+2G,\nabla (\tilde u+2G))=0, \quad\mbox{on}\ [0,T]\times D\\
%&u(T,x)=\Phi(x), \\
%&\frac{\partial u}{\partial\vec{n}}(t,x)-{\red 2\langle g, \vec{n}\rangle}+h(t,x,u)=0, \quad\mbox{on}\ [0,T]\times \partial D,
%\end{split}
%\right.
%\end{equation}

\end{document}